\documentclass[onecolumn,draftclsnofoot]{IEEEtran}
\usepackage{multirow}
\usepackage{graphicx}
\usepackage{amsmath}
\usepackage{amsthm}
\usepackage{mathtools}
\usepackage[psamsfonts]{amssymb}
\usepackage{amsxtra}
\usepackage{hyperref}
\usepackage{paralist}
\usepackage{threeparttable}
\usepackage{cleveref}
\usepackage{cite}
\usepackage{color}
\usepackage[justification=centering]{caption}

\newtheorem{Theorem}{Theorem}[section]
\newtheorem{Lemma}[Theorem]{Lemma}

\newtheorem{Remark}[Theorem]{Remark}
\theoremstyle{definition}
\newtheorem{assump}{Assumption}

\newenvironment{myassump}[2][]
{\renewcommand\theassump{#2}\begin{assump}[#1]}
	{\end{assump}}

\allowdisplaybreaks

\title{\LARGE \bf
Communication-Efficient Distributed Strongly Convex Stochastic Optimization: Non-Asymptotic Rates
}

\author{Anit~Kumar~Sahu, Dusan~Jakovetic, Dragana~Bajovic,  and Soummya~Kar% <-this % stops a space
\thanks{The work of DJ and DB was supported in part by the European Union~(EU) Horizon 2020 project I-BiDaaS, project number~{780787}. The work of D. Jakovetic was also supported in part by the Serbian Ministry of Education, Science, and Technological Development, grant 174030. The work of AKS and SK was supported in part by National Science Foundation under grant CCF-1513936. D. Bajovic is with University of Novi Sad, Faculty of Technical Sciences, Department of Power, Electronics and Communication Engineering 21000 Novi Sad, Serbia {\tt\small dbajovic@uns.ac.rs}. D. Jakovetic is with University of Novi Sad, Faculty of Sciences, Department of Mathematics and Informatics 21000 Novi Sad, Serbia {\tt\small djakovet@uns.ac.rs}.
A. K. Sahu and S. Kar are with the Department of Electrical and Computer Engineering, Carnegie Mellon University, Pittsburgh, PA 15213 {\tt\small \{anits,soummyak\}@andrew.cmu.edu}.}%
}

\begin{document}

\maketitle
\thispagestyle{empty}
\pagestyle{empty}

%%%%%%%%%%%%%%%%%%%%%%%%%%%%%%%%%%%%%%%%%%%%%%%%%%%%%%%%%%%%%%%%%%%%%%%%%%%%%%%%
\begin{abstract}
We examine fundamental tradeoffs in iterative distributed zeroth and first order stochastic optimization in multi-agent networks in terms of \emph{communication cost} (number of per-node transmissions) and \emph{computational cost}, measured by the number of per-node noisy function (respectively, gradient) evaluations with zeroth order (respectively, first order) methods. Specifically, we develop
novel distributed stochastic optimization methods for zeroth and first order strongly convex optimization by utilizing a probabilistic inter-agent communication protocol that increasingly sparsifies communications among agents as time progresses.
Under standard assumptions on the cost functions and the noise statistics, we establish with the proposed method the $O(1/(C_{\mathrm{comm}})^{4/3-\zeta})$ and $O(1/(C_{\mathrm{comm}})^{8/9-\zeta})$ mean square error convergence rates, for the first and zeroth order optimization, respectively, where $C_{\mathrm{comm}}$ is the expected number of network communications and $\zeta>0$ is arbitrarily small. The methods are shown to achieve order-optimal convergence rates in terms of computational cost~$C_{\mathrm{comp}}$, $O(1/C_{\mathrm{comp}})$ (first order optimization) and $O(1/(C_{\mathrm{comp}})^{2/3})$ (zeroth order optimization), while achieving the order-optimal convergence rates in terms of iterations. Experiments on real-life datasets illustrate the efficacy of the proposed algorithms.
\end{abstract}

%%%%%%%%%%%%%%%%%%%%%%%%%%%%%%%%%%%%%%%%%%%%%%%%%%%%%%%%%%%%%%%%%%%%%%%%%%%%%%%%

\section{Introduction}
\label{section-intro}
\noindent Stochastic optimization has taken a central role in problems of learning and inference making over large data sets. Many practical setups are inherently distributed in which, due to sheer data size, it may not be feasible to store data in a single machine or agent. Further, due to the complexity of the objective functions (often, loss functions in the context of learning and inference problems), explicit computation of gradients or exactly evaluating the objective at desired arguments could be computationally prohibitive. The class of stochastic optimization problems of interest can be formalized in the following way:
{\small\begin{align*}
\min f(\mathbf{x}) = \min \mathbb{E}_{\boldsymbol{\xi}\sim P}\left[F(\mathbf{x};\boldsymbol{\xi})\right],
\end{align*}}
\noindent where the information available to implement an optimization scheme usually involves gradients, i.e., $\nabla F(\mathbf{x};\boldsymbol{\xi})$ or function values of $F(\mathbf{x};\boldsymbol{\xi})$ itself. However, both the gradients and the function values are only unbiased estimates of the gradients and the function values of the desired objective $f(\mathbf{x})$. Moreover, due to huge data sizes and distributed applications, the data is often split across different agents, in which case the (global) objective reduces to the sum of $N$ local objectives, $F(\mathbf{x};\boldsymbol{\xi})=\sum_{i=1}^{N}F_{i}(\mathbf{x};\boldsymbol{\xi})$, where $N$ denotes the number of agents. Such kind of scenarios are frequently encountered in setups such as empirical risk minimization in statistical learning\cite{vapnik1998statistical}. In order to address the aforementioned problem setup, we study zeroth and first order distributed stochastic strongly convex optimization over networks.\\ \noindent There are $N$ networked nodes, interconnected through a preassigned possibly sparse communication graph, that collaboratively aim to minimize the sum of their locally known strongly convex costs. We focus on zeroth and first order distributed stochastic optimization methods, where at each time instant (iteration)~$k$, each node queries a stochastic zeroth order oracle~($\mathcal{SZO}$) for a noisy estimate of its local function's value at the current iterate (zeroth order optimization), or a stochastic first order oracle~($\mathcal{SFO}$) for a noisy estimate of its local function's gradient (first order optimization). In both of the proposed stochastic optimization methods, an agent updates its iterate at each iteration by simultaneously assimilating information obtained from the neighborhood~(consensus) and the queried information from the relevant oracle ~(innovations). In the light of the aforementioned distributed protocol, our focus is then on examining the tradeoffs between the \emph{communication cost}, measured by the number of per-node transmissions to their neighboring nodes in the network; and \emph{computational cost}, measured by the number of queries made to $\mathcal{SZO}$ (zeroth order optimization) or $\mathcal{SFO}$ (first order optimization).\\
\noindent \textbf{Contributions}. Our main contributions are as follows. We develop novel methods for zeroth and first order distributed stochastic optimization, based on a probabilistic inter-agent communication protocol that increasingly sparsifies agent communications over time. For the proposed zeroth order method, we establish the $O(1/(C_{\mathrm{comm}})^{8/9-\zeta})$ mean square error (MSE)
convergence rate in terms of communication cost~$C_{\mathrm{comm}}$, where $\zeta>0$ is arbitrarily small. At the same time, the method achieves the order-optimal
$O(1/(C_{\mathrm{comp}})^{2/3})$ MSE rate in terms of computational cost~$C_{\mathrm{comp}}$ in the context of strongly convex functions with second order smoothness. For the first order distributed stochastic optimization, we propose a novel method that is shown to achieve the $O(1/(C_{\mathrm{comm}})^{4/3-\zeta})$ MSE communication rate. At the same time, the proposed method retains the order-optimal $O(1/(C_{\mathrm{comp}}))$ MSE rate in terms of the computational cost, the best achievable rate in the corresponding centralized setting.\\
 \noindent The achieved results reveal an interesting relation between the zeroth and first order distributed stochastic optimization. Namely, as we show here, the zeroth order method achieves a slower MSE communication rate than the first order method due to the (unavoidable) presence of bias in nodes' local functions' gradient estimation. Interestingly, increasing the degree of smoothness\footnote{Degree of smoothness $p$ refers to the function under consideration being $p$-times continuously differentiable with the $p$-th order derivative being Lipschitz continuous.} $p$ in cost functions coupled with a fine-tuned gradient estimation scheme, adapted to the smoothness degree, effectively reduces the bias and enables the zeroth order optimization mean square error to scale as $O(1/(C_{\mathrm{comp}})^{(p-1)/p})$. Thus, with increased smoothness and appropriate gradient estimation schemes, the zeroth order optimization scheme gets increasingly close in mean square error of its first order counterpart. In a sense, we demonstrate that the first order (bias-free) stochastic optimization corresponds to the limiting case of the zeroth order stochastic optimization when $p \rightarrow \infty$. \\
\noindent In more detail, the proposed distributed communication efficient stochastic methods work as follows. They utilize an increasingly sparse communication protocol that we recently proposed in the context of distributed estimation problems \cite{sahu2018communication}. Therein, at each time step (iteration)~$k$,
each node participates in the communication protocol with its
immediate neighbors with a time-decreasing probability~$p_k$.
The probabilities of communicating are equal
across all nodes, while the nodes' decisions whether
to communicate or not are independent of the
past and of the other nodes.
 Upon the transmission stage, if active, each node makes a weighted average
 of its own solution estimate
 and the solution estimates
 received from all
 of its communication-active (transmitting) neighbors,
 assigning to each neighbor a
 time-varying weight~$\beta_k$.
In conjunction with the averaging step,
 the nodes in parallel assimilate the obtained neighborhood information and the local information through a local gradient approximation step -- based on the noisy functions estimates only -- with step-size $\alpha_k$.\\
 \noindent By structure, the proposed distributed zeroth and first order stochastic methods are of a similar
 nature, expect for the fact that
 rather than approximating local gradients based on the noisy functions estimates in the zeroth order case, the first order setup assumes
 noisy gradient estimates are directly available.\\ 
\noindent\textbf{Brief literature review}.
 We now briefly review the literature to help us
contrast this paper from prior work.
 In the context of the extensive literature on distributed
 optimization, the most relevant to our work
 are the references that fall within the following three classes of works: 1)
 distributed strongly convex stochastic optimization;
 2) distributed optimization over random networks (both
 deterministic and stochastic methods); and 3)
 distributed optimization methods
 that aim to improve communication efficiency.
\noindent  While we pursue stochastic optimization in this paper, the case of deterministic noiseless distributed optimization has seen much progress (\cite{BoydADMoM,WotaoYinExtra,WotaoYinDisGrad,SayedExact1}) and more recently accelerated methods (\cite{arxivVersion,UsmanDextra}). For the first class of works,
 several papers give explicit
 convergence rates in terms of the iteration counter~$k$, that here translates
 into computational cost~$C_{\mathrm{comp}}$ or equivalently number of queries to $\mathcal{SZO}$ or $\mathcal{SFO}$, under different assumptions.
 Regarding the underlying network,
 references~\cite{RabbatDistributedStronglyCVX,SayedStochasticOpt}
 consider static networks, while the works~\cite{DistributedMirrorDescent,Kozat,NedicStochasticPush}
  consider deterministic time-varying networks. They all consider
  \emph{first order} optimization.

\noindent References~\cite{RabbatDistributedStronglyCVX,SayedStochasticOpt} consider distributed first order strongly convex optimization for static networks, assuming that the data distributions that underlie each node's local cost function are equal (reference \cite{RabbatDistributedStronglyCVX} considers empirical risks while reference \cite{SayedStochasticOpt} considers risk functions in the form of expectation); this essentially corresponds to each nodes' local function having the same minimizer. References~\cite{DistributedMirrorDescent,Kozat,NedicStochasticPush} consider deterministically varying networks, assuming that the ``union graph'' over finite windows of iterations is connected. The papers~\cite{RabbatDistributedStronglyCVX,SayedStochasticOpt,DistributedMirrorDescent,Kozat} assume undirected networks, while \cite{NedicStochasticPush} allows for
directed networks and assumes a bounded support for the gradient noise. The works~\cite{RabbatDistributedStronglyCVX,DistributedMirrorDescent,Kozat,NedicStochasticPush}
allow the local costs to be non-smooth, while \cite{SayedStochasticOpt} assumes smooth costs, as we do here. With respect to these works, we consider random networks (that are undirected and connected on average), smooth costs, and allow the noise to have unbounded support. The authors of~\cite{hajinezhad2017zeroth} propose a distributed zeroth optimization algorithm for non-convex minimization with a static graph, where a random directions-random smoothing approach was employed.

\noindent For the second class of works, distributed optimization over random networks has been studied in \cite{AsuRandom2,ASU_Math_Prog,randomNesterov}. References \cite{AsuRandom2,ASU_Math_Prog} consider non-differentiable convex costs, first order methods, and no (sub)gradient noise, while reference \cite{randomNesterov} considers differentiable costs with Lipschitz continuous and bounded gradients, first order methods,  and it also does not allow for gradient noise, i.e., it considers methods with exact (deterministic) gradients.
Reference~\cite{CDCSGD2018} considers distributed stochastic first order methods and establishes the method's $O(1/k)$ convergence rate. 
References~\cite{CDCKW2018} considers a zeroth order distributed stochastic approximation method, which queries the $\mathcal{SZO}$ $2d$ times at each iteration where $d$ is the dimension of the optimizer and establishes the method's $O(1/k^{1/2})$ convergence rate in terms of the number of iterations under first order smoothness. 

\noindent In summary, each of the references in the two classes above is not primarily concerned with studying communication rates of distributed stochastic methods. Prior work achieves order-optimal rates in terms of computational cost (that translates here into the number of iterations~$k$), both for the zeroth order, e.g., \cite{CDCKW2018}, and for the first order, e.g., \cite{CDCSGD2018}, distributed strongly convex optimization.\footnote{The works in the first two classes above utilize a non-diminishing amount of communications across iterations, and hence they achieve at best the $O(1/(C_{\mathrm{comm}}))$ (first order optimization) and $O(1/(C_{\mathrm{comm}})^{1/2})$ communication rates.}In contrast, we establish here communication rates as well. This paper and our prior works \cite{CDCKW2018,GlobalSIP2018} distinguish further from other works on distributed zeroth order optimization, e.g.,~\cite{hajinezhad2017zeroth,duchi2015optimal}, in that, not only the gradient is approximated through function values due to the absence of first order information, but also the function values themselves are subject to noise. Reference \cite{GlobalSIP2018} considers a communication efficient zeroth order approximation scheme, where the convergence rate is established to be $O(1/k^{1/2})$ and the MSE-communication is improved to $O(1/(C_{\mathrm{comm}})^{2/3-\zeta})$. In contrast to \cite{GlobalSIP2018}, with additional smoothness assumptions we improve the convergence rate to $O(1/k^{2/3})$ and the MSE-communication is further improved to $O(1/(C_{\mathrm{comm}})^{8/9-\zeta})$.

\noindent Finally, we review the class of works that are concerned with designing distributed methods that achieve communication efficiency, e.g.,~\cite{tsianos2012communication,tsianos2013networked,jakovetic2016distributed,lan2017communication,wang2016decentralized,sahu2018communication,sahu2018communicationNL}. In \cite{wang2016decentralized}, a data censoring method is employed in the context of distributed least squares estimation to reduce computational and communication costs. However, the communication savings in \cite{wang2016decentralized} are a constant proportion with respect to a method which utilizes all communications at all times, thereby not improving the order of the convergence rate. References~\cite{tsianos2012communication,tsianos2013networked,jakovetic2016distributed} also consider a different setup than we do here, namely they study
distributed optimization where the data is available a priori (i.e., it is not streamed). This corresponds to an intrinsically different setting with respect to the one studied here, where actually geometric MSE convergence rates are attainable with stochastic-type methods, e.g., \cite{mokhtari2016dsa}. In terms of the strategy to save communications, references~\cite{tsianos2012communication,tsianos2013networked,jakovetic2016distributed,lan2017communication} consider, respectively, deterministically increasingly sparse communication, an adaptive communication scheme, and selective activation of agents. These strategies are different from ours; we utilize randomized, increasingly sparse communications in general.
In references~\cite{sahu2018communication,sahu2018communicationNL}, we study distributed estimation problems and develop communication-efficient distributed estimators. The problems studied in~\cite{sahu2018communication,sahu2018communicationNL} have a major difference with respect to the current paper in that, in~\cite{sahu2018communication,sahu2018communicationNL}, the assumed setting yields individual nodes' local gradients to evaluate to zero at the global solution. In contrast, the model assumed here does not feature such property, and hence it is more challenging. 

\noindent Finally, we comment on the recent paper~\cite{lan2017communication} that develops communication-efficient distributed methods for both non-stochastic and stochastic distributed first order optimization, both in the presence and in the absence of the strong convexity assumption. For the stochastic, strongly convex first order optimization, \cite{lan2017communication} shows that the method therein gets $\epsilon$-close to the solution in $O(1/\sqrt{\epsilon})$ communications and with an $O(1/\epsilon)$ computational cost. The current paper has several differences
with respect to~\cite{lan2017communication}. First, reference~\cite{lan2017communication} does not study zeroth order optimization. Second, this work
assumes for the gradient noise to be independent of the algorithm iterates. This is a strong assumption that may be not satisfied, e.g., with many machine learning applications.
Third, while we assume here twice differentiable costs, this assumption is not imposed in~\cite{lan2017communication}. Finally, the method in~\cite{lan2017communication} is considerably more complex than the one proposed here, with two layers of iterations (inner and outer iterations). In particular, the inner iterations involve solving an exact minimization problem which necessarily points to the usage of an off-the-shelf solver, the computation cost of which is not factored into the computation cost in \cite{lan2017communication}.

\noindent \textbf{Paper organization}. The next paragraph introduces notation.
Section~\ref{sec:model} describes the model and the proposed algorithms for zeroth and first order
distributed stochastic optimization. Section~\ref{sec:main_results} states our convergence rates results for the two methods.
Sections~\ref{sec:proof_main_res_0} and~\ref{sec:proof_main_res_1}  provide proofs for the zeroth and first order methods,
respectively. Section~\ref{sec:sim} demonstrates communication efficiency
of the proposed methods through numerical examples. Finally,
we conclude in Section~\ref{sec:conc}.
%The next
%paragraph introduces notation.
%Section~2 describes the model and the stochastic gradient method we consider.
%Section~3 states and proves the main result on the algorithm's MSE convergence rate.
%Section~4 provides a simulation example. Finally, we conclude in Section~5.

\noindent\textbf{Notation}.
We denote by $\mathbb R$ the set of real numbers and by ${\mathbb R}^m$ the $m$-dimensional
Euclidean real coordinate space. We use normal lower-case letters for scalars,
lower case boldface letters for vectors, and upper case boldface letters for
matrices. Further, we denote by: $\mathbf{A}_{ij}$ the entry in the $i$-th row and $j$-th column of
a matrix $\mathbf{A}$;
$\mathbf{A}^\top$ the transpose of a matrix $A$; $\otimes$ the Kronecker product of matrices;
$I$, $0$, and $\mathbf{1}$, respectively, the identity matrix, the zero matrix, and the column vector with unit entries; $\mathbf{J}$ the $N \times N$ matrix $J:=(1/N)\mathbf{1}\mathbf{1}^\top$.
When necessary, we indicate the matrix or vector dimension as a subscript.
Next, $A \succ  0 \,(A \succeq  0 )$ means that
the symmetric matrix $A$ is positive definite (respectively, positive semi-definite). For a set $\mathcal{X}$, $|\mathcal{X}|$ denotes the cardinality of set $\mathcal{X}$.
We denote by:
$\|\cdot\|=\|\cdot\|_2$ the Euclidean (respectively, induced) norm of its vector (respectively, matrix) argument; $\lambda_i(\cdot)$, the $i$-th smallest eigenvalue of its matrix argument; $\nabla h(w)$ and $\nabla^2 h(w)$ the gradient and Hessian, respectively, evaluated at $w$ of a function $h: {\mathbb R}^m \rightarrow {\mathbb R}$, $m \geq 1$; $\mathbb P(\mathcal A)$ and $\mathbb E[u]$ the probability of
an event $\mathcal A$ and expectation of a random variable $u$, respectively. By $\mathbf{e}_{j}$ we denote the $j$-th column of the identity matrix $\mathbf{I}$ where the dimension is made clear from the context.
Finally, for two positive sequences $\eta_n$ and $\chi_n$, we have: $\eta_n = O(\chi_n)$ if
$\limsup_{n \rightarrow \infty}\frac{\eta_n}{\chi_n}<\infty$.

\section{Model and the proposed algorithms}
\label{sec:model}
\noindent The network of $N$ agents in our setup collaboratively aim to solve the following unconstrained problem:
{\small\begin{align}
	\label{eq:opt_problem}
	\min_{\mathbf{x}\in\mathbb{R}^{d}}\sum_{i=1}^{N}f_{i}(\mathbf{x}),
	\end{align}}
where $f_{i}: \mathbb{R}^{d}\mapsto\mathbb{R}$ is a strongly convex
function available to node $i$, $i=1,...,N$. We make the following assumption on the functions $f_{i}(\cdot)$:

\begin{myassump}{1}
	\label{as:1}
	For all $i=1,...,N$, function
	$f_{i}: \mathbb{R}^{d}\mapsto\mathbb{R}$ is twice continuously differentiable with Lipschitz continuous gradients. In particular, there exist constants $L, \mu > 0$ such that for all $\mathbf{x} \in {\mathbb R}^d$, 
	{\small\begin{align*}
		\mu\ \mathbf{I} \preceq \nabla^2 f_i(\mathbf{x}) \preceq L~\mathbf{I}.
		\end{align*}}
\end{myassump}

\noindent From Assumption~\ref{as:1} we have that each $f_i$, $i=1,\cdots,N$, is $\mu$-strongly convex. Using standard properties of strongly convex functions, we have for any $\mathbf{x,y} \in {\mathbb R}^d$:
{\small\begin{align*}
	&f_i(\mathbf{y}) \geq f_i(\mathbf{x})+ \nabla f_i(\mathbf{x})^\top \,(\mathbf{y}-\mathbf{x})+ \frac{\mu}{2}\|\mathbf{x}-\mathbf{y}\|^2,\\
	&\|\nabla f_i(\mathbf{x}) -\nabla f_i(\mathbf{y})\| \leq L\,\|\mathbf{x}-\mathbf{y}\|.
	\end{align*}}
\noindent We also have that from assumption \ref{as:1}, the optimization problem in \eqref{eq:opt_problem} has a unique solution, which we denote by $\mathbf{x}^{\ast}\in\mathbb{R}^{d}$. Throughout the paper, we use the sum function which is defined as $f:\,{\mathbb R}^d \rightarrow \mathbb R$,~$f(\mathbf{x})=\sum_{i=1}^N f_i(\mathbf{x})$. We consider distributed stochastic
gradient methods to solve~\eqref{eq:opt_problem}.
That is, we study algorithms of the following form:
{\small\begin{align}
\label{eq:opt_strategy}
&\mathbf{x}_i(k+1)=\mathbf{x}_i(k) - \sum_{j \in \Omega_i(k)}\gamma_{i,j}(k)\left( \mathbf{x}_i(k) - \mathbf{x}_j(k) \right)\nonumber\\
&-\alpha_k\widehat{\mathbf{g}}_{i}(\mathbf{x}_{i}(k)),
\end{align}}
where the weight assigned to an incoming message $\gamma_{i,j}(k)$ and the neighborhood of an agent $\Omega_i(k)$ are determined by the specific instance of the designated communication protocol. The approximated gradient $\widehat{\mathbf{g}}_{i}(\mathbf{x}_{i}(k))$ is specific to the optimization, i.e., whether it is a zeroth order optimization or a first order optimization scheme. Technically speaking, as we will see later, a zeroth order optimization scheme approximates the gradient as a biased estimate of the gradient while a first order optimization scheme approximates the gradient as an unbiased estimate of the gradient. The variation in the gradient approximation across first order and zeroth order methods can be attributed to the fact that the oracles from which the agents query for information pertaining to the loss function differ. For instance, in the case of the zeroth order optimization, the agents query a stochastic zeroth order oracle~($\mathcal{SZO}$) and in turn receive noisy function values~(unbiased estimates) for the queried point. However, in the case of first order optimization, the agents query a stochastic first order oracle~($\mathcal{SFO}$) and receive unbiased estimates of the gradient. In subsequent sections, we will explore the gradient approximations in greater detail.
Before stating the algorithms, we first discuss the communication scheme.
\noindent Specifically,
we adopt the following model.
\noindent\subsubsection{Communication Scheme}
\noindent The inter-node communication network to which the information exchange between nodes conforms to is modeled as an \emph{undirected} simple connected graph $G=(V,E)$, with $V=\left[1\cdots N\right]$ and~$E$ denoting the set of nodes and communication links. The neighborhood of node~$n$ is given by $\Omega_{n}=\left\{l\in V\,|\,(n,l)\in E\right\}$. The node~$n$ has degree $d_{n}=|\Omega_{n}|$. The structure of the graph is described by the  $N\times N$ adjacency matrix, $\mathbf{A}=\mathbf{A}^\top=\left[\mathbf{A}_{nl}\right]$, $\mathbf{A}_{nl}=1$, if $(n,l)\in E$, $\mathbf{A}_{nl}=0$, otherwise. The graph Laplacian $\overline{\mathbf{R}}=\mathbf{D}-\mathbf{A}$ is positive semidefinite, with eigenvalues ordered as $0=\lambda_{1}(\overline{\mathbf{R}}) \leq\lambda_{2}(\overline{\mathbf{R}}) \leq\cdots \leq \lambda_{N}(\overline{\mathbf{R}})$, where $\mathbf{D}$ is given by $\mathbf{D}=\mbox{diag}\left(d_{1}\cdots d_{N}\right)$. 
\noindent We make the following assumption on $\overline{\mathbf{R}}$.
\begin{myassump}{2}
	\label{as:2}
	\label{assumption-network}
	The inter-agent communication graph is connected on average, i.e., $\overline{\mathbf{R}}$ is connected. In other words, $\lambda_{2}(\mathbf{\overline{\mathbf{R}}}) > 0$.
\end{myassump}
\noindent Thus, $\overline{\mathbf{R}}$ corresponds to the maximal graph, i.e., the graph of all \emph{allowable} communications. We now describe our randomized communication protocol that selects a (random) subset of the allowable links at each time instant for information exchange.\\
\noindent For each node $i$, at every time $k$, we introduce a binary random variable $\psi_{i,k}$, where
{\small\begin{align}
	\label{eq:dec_rule_lin_1}
	\psi_{i,k}=
	\begin{cases}
	\rho_{k} &~~\textit{with~probability}~\zeta_{k}\\
	0 & ~~\textit{otherwise},
	\end{cases}
	\end{align}}
where $\psi_{i,k}$'s are independent both across time and the nodes, i.e., across $k$ and $i$ respectively. The random variable $\psi_{i,k}$ abstracts out the decision of the node $i$ at time $k$ whether to participate in the neighborhood information exchange or not. We specifically take $\rho_{k}$ and $\zeta_{k}$ of the form
{\small\begin{align}
	\label{eq:time_decay}
	\rho_{k} = \frac{\rho_{0}}{(k+1)^{\epsilon/2}},~\zeta_{k} = \frac{\zeta_{0}}{(k+1)^{(\tau/2-\epsilon/2)}},
	\end{align}}
where $0<\tau\leq \frac{1}{2}$ and $0<\epsilon<\tau$. Furthermore, define $\beta_{k}$ to be
{\small\begin{align}
	\label{eq:beta}
	\beta_{k}=\left(\rho_{k}\zeta_{k}\right)^{2} = \frac{\beta_{0}}{(k+1)^{\tau}},
	\end{align}}
\noindent where $\beta_0=\rho_0^2\zeta_0^2$. 
\noindent With the above development in place, we define the random time-varying Laplacian $\mathbf{R}(k)$, where $\mathbf{R}(k)\in\mathbb{R}^{N\times N}$ abstracts the inter-node information exchange as follows:
{\small\begin{align}
\label{eq:laplacian}
\mathbf{R}_{i,j}(k)=
\begin{cases}
-\psi_{i,k}\psi_{j,k} & \{i,j\}\in E, i\neq j\\
0 & i\neq j, \{i,j\}\notin E\\
\sum_{l\neq i}\psi_{i,k}\psi_{l,k}& i=j.
\end{cases}
\end{align}}
\noindent The above communication protocol allows two nodes to communicate only when the link is established in a bi-directional fashion and hence avoids directed graphs. The design of the communication protocol as depicted in \eqref{eq:dec_rule_lin_1}-\eqref{eq:laplacian} not only decays the weight assigned to the links over time but also decays the probability of the existence of a link. Such a design is consistent with frameworks where the agents have finite power and hence not only the number of communications, but also, the quality of the communication decays over time.
\noindent We have, for $\{i,j\}\in E$ and $i\neq j$: 
{\small\begin{align}
\label{eq:expectation}
&\mathbb{E}\left[\mathbf{R}_{i,j}(k)\right]= -\left(\rho_{k}\zeta_{k}\right)^{2} = -\beta_{k} = -\frac{\beta_0}{(k+1)^{\tau}}\nonumber\\
&\mathbb{E}\left[\mathbf{R}_{i,j}^{2}(k)\right] = \left(\rho_{k}^{2}\zeta_{k}\right)^{2} = \frac{\rho_0^{2}\beta_0}{(k+1)^{\tau+\epsilon}}.
\end{align}}
\noindent Thus, we have that, the variance of $\mathbf{R}_{i,j}(k)$ is given by,
{\small\begin{align}
\label{eq:variance}
\textit{Var}\left(\mathbf{R}_{i,j}(k)\right) = \frac{\beta_{0}\rho_{0}^{2}}{(k+1)^{\tau+\epsilon}} - \frac{\beta_0^{2}}{(k+1)^{2\tau}}.
\end{align}}
\noindent Define, the mean of the random time-varying Laplacian sequence $\{\mathbf{R}(k)\}$ as $\overline{\mathbf{R}}_{k} = \mathbb{E}\left[\mathbf{R}(k)\right]$ and $\widetilde{\mathbf{R}}(k) = \mathbf{R}(k)-\overline{\mathbf{R}}_{k}$.
\noindent Note that, $\mathbb{E}\left[\widetilde{\mathbf{R}}(k)\right] = \mathbf{0}$, and
{\small\begin{align}
\label{eq:laplace_res}
\mathbb{E}\left[\left\|\widetilde{\mathbf{R}}(k)\right\|^{2}\right] \leq 4N^{2}\mathbb{E}\left[\widetilde{\mathbf{R}}_{i,j}^{2}(k)\right] = \frac{4N^{2}\beta_{0}\rho_{0}^{2}}{(k+1)^{\tau+\epsilon}} - \frac{4N^{2}\beta_0^{2}}{(k+1)^{2\tau}},
\end{align}}
where $\left\|\cdot\right\|$ denotes the $\mathcal{L}_{2}$ norm. The above equation follows by relating the $\mathcal{L}_{2}$ and Frobenius norms.\\
\noindent We also have that, $\overline{\mathbf{R}}_{k}=\beta_{k}\overline{\mathbf{R}}$, where
{\small\begin{align}
\label{eq:laplacian1}
\overline{\mathbf{R}}_{i,j}=
\begin{cases}
-1 & \{i,j\}\in E, i\neq j\\
0 & i\neq j, \{i,j\}\notin E\\
-\sum_{l\neq i}\overline{\mathbf{R}}_{i,l}& i=j.
\end{cases}
\end{align}}
\noindent Technically speaking, the communication graph at each time $k$ encapsulated as $\mathbf{R}(k)$ need not be connected at all times, although the graph of allowable links $G$ is connected.. In fact, at any given time $k$, only a few of the possible links could be active. However, since $\overline{\mathbf{R}}_{k}=\beta_{k}\overline{\mathbf{R}}$, we note that, by Assumption \ref{as:2}, the instantaneous Laplacian $\mathbf{R}(k)$ is connected on average.The connectedness in average basically ensures that over time, the information from each agent in the graph reaches other agents over time in a symmetric fashion and thus ensuring information flow, while providing the leeway for the instantaneous communication graphs at different times to be not connected.\\
We employ a primal algorithm for solving the optimization problem in  \eqref{eq:opt_problem}. In particular, the update in \eqref{eq:opt_strategy} can then be written in a vector form as follows:
{\small\begin{align}
\label{eq:opt_strategy_1}
\mathbf{x}(k+1) = \mathbf{W}_{k}\mathbf{x}(k)-\alpha_{k}\widehat{\mathbf{G}}(\mathbf{x}(k)),
\end{align}}
where {\small$\mathbf{x}(k) = \left[\mathbf{x}_{1}^{\top}(k),\cdots,\mathbf{x}_{N}^{\top}(k)\right]^{\top}\in\mathbb{R}^{Nd}$, $F(\mathbf{x})=\sum_{i=1}^{N}f_{i}(\mathbf{x}_{i})$, $\mathbf{x} = \left[\mathbf{x}_{1}^{\top},\cdots,\mathbf{x}_{N}^{\top}\right]^{\top}\in\mathbb{R}^{Nd}$, $\widehat{\mathbf{G}}(\mathbf{x}(k))=[\widehat{\mathbf{g}}_{i}^{\top}(\mathbf{x}_{i}(k)),\cdots,\widehat{\mathbf{g}}_{}^{\top}(\mathbf{x}_{N}(k))]^{\top}$} and {\small$\mathbf{W}_{k}=\left(\mathbf{I}-\mathbf{R}(k)\right)\otimes \mathbf{I}_{d}$}.
\noindent We state an assumption on the weight sequences before proceeding further.
\begin{myassump}{3}
	\label{as:3}
	The weight sequence $\alpha_{k}$ is given by $\alpha_0/(k+1)$, where $\alpha_{0}>1/\mu$. For the sequence $\rho_{k}$ as defined in \eqref{eq:time_decay}, it is chosen in such a way that,
	{\small\begin{align}
	\label{eq:as51}
	\rho_{0}^{2} \le \frac{4N^{2}}{\lambda_{2}\left(\overline{\mathbf{R}}\right)}.
	\end{align}}
\end{myassump}
In the following sections, we propose two different approaches to solve the optimization problem in \eqref{eq:opt_problem}. The first approach involves zeroth order optimization, while the second approach involves a first order optimization. We first study the zeroth order approach to the problem in  \eqref{eq:opt_problem}.
\subsection{Zeroth Order Optimization}
\label{sec:zero_opt}
\noindent We employ a random directions stochastic approximation~(RDSA) type method from \cite{nesterov2011random} adapted to our distributed setup to solve~\eqref{eq:opt_problem}. Each node~$i$, $i=1,...,N$, in our setup maintains a local copy of its local estimate of the optimizer $\mathbf{x}_i(k) \in {\mathbb R}^d$ at all times. In addition to the smoothness assumption in \ref{as:1}, we define additional smoothness assumptions in the context of zeroth order optimization. 
\begin{myassump}{A1}
	\label{as:a1}
	For all $i=1,...,N$, the functions
	$f_{i}: \mathbb{R}^{d}\mapsto\mathbb{R}$ have their Hessian to be $M$-Lipschitz, i.e.,
	{\small\begin{align*}
		\|\nabla^{2} f_i(\mathbf{x}) -\nabla^{2} f_i(\mathbf{y})\| \leq M\,\|\mathbf{x}-\mathbf{y}\|, \forall i=1,\cdots,N.
		\end{align*}}
\end{myassump}

\noindent In order to carry out the optimization, each agent $i$ makes queries to the $\mathcal{SZO}$ at time $k$, from which the agent obtains noisy function values of $f_i(\mathbf{x}_i(k))$. Denote the noisy value of $f_{i}(\cdot)$ as $\widehat{f}_{i}(\cdot)$ where,
{\small\begin{align}
\label{eq:func_dist}
\widehat{f}_{i}(\mathbf{x}_{i}(k))=f_{i}(\mathbf{x}_{i}(k))+\widehat{v}_{i}(k;\mathbf{x}_{i}(k)),
\end{align}}
\noindent where the first argument in $\widehat{v}_{i}(k;\mathbf{x}_{i}(k))$ is the iteration number, and the second argument is the point at which the $\mathcal{SZO}$ oracle is queried. The properties of the noise $\widehat{v}_{i}(k;\mathbf{x}_{i}(k))$ are discussed further ahead.
\noindent Typically due to the unavailability of the analytic form of the functionals in zeroth order methods, the gradient cannot be explicitly evaluated and hence, we resort to a gradient approximation. In order to approximate the gradient, each agent makes three calls to the stochastic zeroth order oracle. For instance, agent $i$ queries for {\small$f_i(\mathbf{x}_i(k)+c_{k}\mathbf{z}_{i,k})$, $f_i(\mathbf{x}_i(k)+c_{k}\mathbf{z}_{i,k}/2)$} and $f_i(\mathbf{x}_i(k))$ at time $k$ and obtains {\small$\widehat{f}_i(\mathbf{x}_i(k)+c_{k}\mathbf{z}_{i,k})$, $\widehat{f}_i(\mathbf{x}_i(k)+c_{k}\mathbf{z}_{i,k}/2)$} and $\widehat{f}_i(\mathbf{x}_i(k))$ respectively, where $c_{k}$ is a carefully chosen time-decaying constant and $\mathbf{z}_{i,k}$ is a random vector~(to be specified soon) such that $\mathbb{E}\left[\mathbf{z}_{i,k}\mathbf{z}_{i,k}^{\top}\right]=\mathbf{I}_{d}$.\\
\noindent Denote by $\widehat{\mathbf{g}_{i}}(\mathbf{x}_{i}(k))$ the approximated gradient which is given by:
{\small\begin{align}
\label{eq:KW_grad}
&\widehat{\mathbf{g}_{i}}(\mathbf{x}_{i}(k)) \doteq  2\widetilde{\mathbf{g}_{i}}\left(\mathbf{x}_{i}(k),\frac{c_{k}}{2}\right)-\widetilde{\mathbf{g}_{i}}\left(\mathbf{x}_{i}(k),c_{k}\right)\nonumber\\
&=\frac{4\widehat{f}_{i}\left(\mathbf{x}_i(k)+\frac{c_{k}}{2}\mathbf{z}_{i,k}\right)-4\widehat{f}_{i}\left(\mathbf{x}_{i}(k)\right)}{c_{k}}\mathbf{z}_{i,k}\nonumber\\&-\frac{\widehat{f}_{i}\left(\mathbf{x}_i(k)+c_{k}\mathbf{z}_{i,k}\right)-\widehat{f}_{i}\left(\mathbf{x}_{i}(k)\right)}{c_{k}}\mathbf{z}_{i,k},
\end{align}}
\noindent where $\widetilde{\mathbf{g}_{i}}\left(\cdot,\cdot\right)$ represents a first order finite difference operation and $\theta_1,\theta_2\in[0,1]$. Note that, the gradient approximation derived in \eqref{eq:KW_grad} involves the noise in the retrieved function value from the $\mathcal{SZO}$ differently from other RDSA approaches such as in \cite{duchi2015optimal,nesterov2011random}. The finite difference technique used in \eqref{eq:KW_grad} resembles, \emph{the twicing trick} commonly used in Kernel density estimation which is employed so as to reduce bias and approximately eliminate the effect of the second degree term from the bias. It is also to be noted that the number of queries made to the $\mathcal{SZO}$ at every gradient approximation is $3$.
\noindent Thus, we can write,
{\small\begin{align}
\label{eq:kW_grad_1}
&\widehat{\mathbf{g}_{i}}(\mathbf{x}_{i}(k)) = \nabla f_{i}\left(\mathbf{x}_{i}(k)\right)+\underbrace{\mathbb{E}\left[\widehat{\mathbf{g}_{i}}(\mathbf{x}_{i}(k))|\mathcal{F}_{k}\right]-\nabla f_{i}\left(\mathbf{x}_{i}(k)\right)}_{\text{$c_{k}\mathbf{b}_{i}(\mathbf{x}_{i}(k))$}}\nonumber\\&+\underbrace{\mathbf{g}_{i}(\mathbf{x}_{i}(k))-\mathbb{E}\left[\widehat{\mathbf{g}_{i}}(\mathbf{x}_{i}(k))|\mathcal{F}_{k}\right]+\frac{v_{i}(k;\mathbf{x}_{i}(k))\mathbf{z}_{i,k}}{c_{k}}}_{\text{$\mathbf{h}_{i}(\mathbf{x}_{i}(k))$}},
\end{align}}
where 
{\small\begin{align}
	\label{eq:kw_grad_1.5}
	&\mathbf{g}_{i}(\mathbf{x}_{i}(k))=\frac{4f_{i}\left(\mathbf{x}_i(k)+\frac{c_{k}}{2}\mathbf{z}_{i,k}\right)-4f_{i}\left(\mathbf{x}_{i}(k)\right)}{c_{k}}\mathbf{z}_{i,k}\nonumber\\&-\frac{f_{i}\left(\mathbf{x}_i(k)+c_{k}\mathbf{z}_{i,k}\right)-f_{i}\left(\mathbf{x}_{i}(k)\right)}{c_{k}}\mathbf{z}_{i,k},
	\end{align}}
{\small\begin{align}
\label{eq:kw_grad_2}
&v_{i}(k;\mathbf{x}_{i}(k)) =4\left(\widehat{f}_i\left(\mathbf{x}_i(k)+\frac{c_{k}}{2}\mathbf{z}_{i,k}\right)-f_i\left(\mathbf{x}_i(k)+\frac{c_{k}}{2}\mathbf{z}_{i,k}\right)\right)\nonumber\\&-3(\widehat{f}_{i}(\mathbf{x}_{i}(k))-f_{i}(\mathbf{x}_{i}(k)))-(\widehat{f}_i\left(\mathbf{x}_i(k)+c_{k}\mathbf{z}_{i,k}\right)\nonumber\\&-f_i\left(\mathbf{x}_i(k)+c_{k}\mathbf{z}_{i,k}\right)),
\end{align}}
and,
\noindent $\mathcal{F}_k$ denotes the history of the proposed algorithm up to time $k$. Given that the sources of randomness in our algorithm are the noise sequence $\{\mathbf{v}(k;\mathbf{x}(k))\}$, the random network sequence $\{\mathbf{R}(k)\}$ and the random vectors for directional derivatives $\{\mathbf{z}_{k}\}$, $\mathcal{F}_k$ is given by the $\sigma$-algebra generated by the collection of random variables {\small$\{\,\mathbf{R}(s),\,\mathbf{v}(k;\mathbf{x}(k)),\,\mathbf{z}_{i,s}\}$,~$i=1,...,N$,~$s=0,...,k-1$}.\\
\noindent In general, the higher order smoothness imposed by Assumption \ref{as:3} allows us to use a higher order finite difference approximation for estimating the gradient. Due to assumption \ref{as:3}, the bias in the gradient estimate by employing a second order finite difference approximation of the gradient is of the order $O(c_{k}^{2})$. Instead, a first order finite difference approximation of the gradient would have yielded a bias of $O(c_{k})$. More generally, an assumption involving $p$-th order smoothness of the loss functions would have enabled usage of a $p$-th degree finite difference approximation of the gradient thus leading to a bias of $O(c_{k}^{p})$.
\begin{myassump}{A2}
	\label{as:a2}
	The $z_{i,k}$'s are drawn from a distribution $P$ such that {\small$\mathbb{E}\left[\mathbf{z}_{i,k}\mathbf{z}_{i,k}^{\top}\right]=\mathbf{I}_{d}$,  $s_{1}(P)=\mathbb{E}\left[\left\|\mathbf{z}_{i,k}\right\|^{4}\right]$} and {\small$s_{2}(P)=\mathbb{E}\left[\left\|\mathbf{z}_{i,k}\right\|^{6}\right]$} are finite. 
	%and for any vector $\mathbf{g}\in\mathbb{R}^{d}$, there exists a function $s(d):\mathbb{N}\mapsto\mathbb{R}_{+}$ such that,
	%{\small\begin{align*}
	%\mathbb{E}\left[\left\|\langle g,\mathbf{z}_{i,k}\rangle\mathbf{z}_{i,k}\right\|^{2}\right] \le s(d)\left\|g\right\|^{2}.
	%\end{align*}}
\end{myassump}
\noindent We provide two examples of two such distributions. If $\mathbf{z}_{i,k}$'s are drawn from $\mathcal{N}(0,\mathbf{I}_{d})$, then {\small$\mathbb{E}\left[\left\|\mathbf{z}_{i,k}\right\|^{4}\right]=d(d+2)$} and {\small$\mathbb{E}\left[\left\|\mathbf{z}_{i,k}\right\|^{6}\right]=d(d+2)(d+4)$}. If $\mathbf{z}_{i,k}$'s are drawn uniformly from the $l_{2}$-ball of radius $\sqrt{d}$, then we have, $\left\|\mathbf{z}_{i,k}\right\| = \sqrt{d}$, {\small$\mathbb{E}\left[\left\|\mathbf{z}_{i,k}\right\|^{4}\right] = d^{2}$} and {\small$\mathbb{E}\left[\left\|\mathbf{z}_{i,k}\right\|^{4}\right] = d^{3}$}. For the rest of the paper, we assume that $\mathbf{z}_{i,k}$'s are sampled from a normal distribution with {\small$\mathbf{E}\left[\mathbf{z}_{i,k}\mathbf{z}_{i,k}^{\top}\right] = \mathbf{I}_{d}$} or uniformly from the surface of the $l_2$-ball of radius $\sqrt{d}$.
%Furthermore, we have that, $\mathbb{E}\left[\left\|\langle g,\mathbf{z}_{i,k}\rangle\mathbf{z}_{i,k}\right\|^{2}\right] = d\left[\langle g,\mathbf{z}_{i,k}\rangle^{2}\right] = d\left\|g\right\|^{2}$, by utilizing the fact that $\mathbb{E}\left[\mathbf{z}_{i,k}\mathbf{z}_{i,k}^{\top}\right]=\mathbf{I}_{d}$.
\begin{Remark}
	\label{rm:1}
	The RDSA scheme~(see, for example \cite{nesterov2011random}) used here is similar to the simultaneous perturbation stochastic approximation scheme~(SPSA) as proposed in \cite{spall1992multivariate}. In SPSA, each dimension $i$ of the optimization iterate is perturbed by a random variable $\Delta_{i}$. However, instead of RDSA where the directional derivative is taken along the sampled vector $\mathbf{z}$, the directional derivative in case of SPSA is along the direction $\left[1/\Delta_{1},\cdots,1/\Delta_{d}\right]$ which thus needs boundedness of the inverse moments of the random variable $\Delta_{i}$. The particular choice for $\Delta_{i}$'s is taken to be the Bernoulli distribution with $\Delta_{i}$'s taking values $1$ and $-1$ with probability $0.5$. It is to be noted that at each iteration, both RDSA and SPSA approximate the gradient by making two calls to the stochastic zeroth order oracle as opposed to $d$ calls in the case of Kiefer Wolfowitz Stochastic Approximation~(KWSA)~(see, \cite{kiefer1952stochastic} for example).
\end{Remark}

\noindent For arbitrary deterministic initializations $\mathbf{x}_i(0) \in {\mathbb R}^d$, $i=1,...,N$, the optimizer update rule at node $i$ and $k=0,1,...,$ is given as follows:
{\small\begin{align}
\label{eq:update_rule_node}
&\mathbf{x}_i(k+1)=\mathbf{x}_i(k) - \sum_{j \in \Omega_i(k)}\psi_{i,k}\psi_{j,k}\left( \mathbf{x}_i(k) - \mathbf{x}_j(k) \right)\nonumber\\
&-\alpha_k\widehat{\mathbf{g}}_{i}(\mathbf{x}_{i}(k)),
\end{align}}
\noindent where $\widehat{\mathbf{g}}_{i}(\cdot)$ is as defined in \eqref{eq:kW_grad_1}. Comparing to the general update in \eqref{eq:opt_strategy}, the time-varying weight $\gamma_{i,j}(k)$ at agent $i$ to the incoming message from agent $j$ is given by $\psi_{j,k}$.
\begin{Remark}
	\label{rm:2}
	The main intuition behind the randomized activation albeit in a controlled manner for both the zeroth order and first order optimization methods is the fact that in expectation both the updates exactly reduce to the update where the communication graph between agents is realized by the expected Laplacian. 
\end{Remark}
\noindent It is to be noted that unlike first order stochastic gradient methods, where the algorithm has access to unbiased estimates of the gradient, the local gradient estimates $\mathbf{g}_{i}(\cdot)$ used in \eqref{eq:update_rule_node} are biased (see \eqref{eq:kW_grad_1}) due to the unavailability of the exact gradient functions and their approximations using the zeroth order scheme in \eqref{eq:KW_grad}. The update is carried on in all agents parallely in a synchronous fashion. The weight sequences $\{\alpha_{k}\}$, $\{c_{k}\}$ and $\{\beta_{k}\}$ are given by $\alpha_{k}=\alpha_0/(k+1)$, $c_{k}=c_0/(k+1)^{\delta}$ and $\beta_{k}=\beta_0/(k+1)^{\tau}$ respectively, where $\alpha_0, c_0, \beta_0 > 0$. We state an assumption on the weight sequences before proceeding further.
\begin{myassump}{A3}
	\label{as:a3}
	The sequence $c_{k}$ is given by:
	{\small\begin{align}
	\label{eq:c_k}
	c_{k} = \frac{1}{s_{1}(P)(k+1)^{\delta}},
	\end{align}}
	\noindent where $\delta > 0$.
	The constant $\delta > 0$ is chosen in such a way that,
	{\small\begin{align}
	\label{eq:a3}
	\sum_{k=1}^{\infty} \frac{\alpha_{k}^2}{c_k^2} < \infty
	\end{align}}
\end{myassump}
\noindent The update in \eqref{eq:update_rule_node} can be written as:
{\small\begin{align}
\label{eq:0update_rule}
&\mathbf{x}(k+1) = \mathbf{W}_{k}\mathbf{x}(k)-\alpha_{k}\nabla F(\mathbf{x}(k))-\alpha_{k}c_{k}\mathbf{b}(\mathbf{x}(k))\nonumber\\&-\alpha_{k}\mathbf{h}(\mathbf{x}(k)),
\end{align}}
where {\small$\mathbf{b}(\mathbf{x}(k)) = \left[\mathbf{b}_{1}^{\top}\left(\mathbf{x}_{1}(k)\right),\cdots,\mathbf{b}_{N}^{\top}\left(\mathbf{x}_{N}(k)\right)\right]^{\top}\in\mathbb{R}^{Nd}$} and {\small$\mathbf{h}(\mathbf{x}(k)) = \left[\mathbf{h}_{1}^{\top}\left(\mathbf{x}_{1}(k)\right),\cdots,\mathbf{h}_{N}^{\top}\left(\mathbf{x}_{N}(k)\right)\right]^{\top}\in\mathbb{R}^{Nd}$}.
We state an assumption on the measurement noises next.
\begin{myassump}{A4}
	\label{as:a4}
	For each $i=1,...,N$,
	the sequence of measurement noises $\{v_i(k;\mathbf{x}_{i}(k))\}$
	satisfies for all $k=0,1,...$:
	{\small\begin{align}
	\label{eq:as4}
	&\mathbb{E}[\,v_i(k;\mathbf{x}_{i}(k))\,|\,\mathcal{F}_k, \mathbf{z}_{i,k}] =0,\,\,\mathrm{almost\,surely\,(a.s.)}\nonumber\\
	&\mathbb{E}[\,v_i(k;\mathbf{x}_{i}(k))^2\,|\,\mathcal{F}_k,\mathbf{z}_{i,k}] \leq c_{v}\|\mathbf{x}_i(k)\|^2+\sigma_{v}^{2},
	\,\,\mathrm{a.s.},
	\end{align}}
	where $c_v$ and $\sigma_{v}^{2}$ are nonnegative constants. 
	%Furthermore, the measurement noise at each agent $i$ at each time $k$, is conditionally independent of the direction $\mathbf{z}_{i,k}$ given the history $\mathcal{F}_{k}$, i.e.,
	%{\small\begin{align}
	%\label{eq:0noise_1}
	%\mathbb{E}\left[v_{i}(k;\mathbf{x}_{i}(k))\mathbf{z}_{i,k}|\mathcal{F}_{k}\right] = %\mathbb{E}\left[\mathbf{z}_{i,k}|\mathbb{E}\left[v_{i}(k;\mathbf{x}_{i}(k))|\mathcal{F}_{k},\mathbf{z}_{i,k}\right]\mathcal{F}_{k}\right].
	%\end{align}}
\end{myassump}
\noindent Assumption \ref{as:a4} is standard in the analysis of stochastic optimization methods, e.g., \cite{SayedStochasticOpt}. It is stated in terms of noise $\mathbf{v}_i(k;\mathbf{x}_{i}(k))$ in \eqref{eq:kw_grad_2} rather then directly in terms of the $\mathcal{SZO}$ noises in equation \eqref{eq:func_dist}, for notational simplicity. An equivalent statement can be made in terms of the noises in \eqref{eq:func_dist}. The assumption about the conditional independence between the random directions $\mathbf{z}_{i,k}$ and the function noise $v_{i}(k;\mathbf{x}_{i}(k))$ is mild. It merely formalizes the model that we consider, namely that, given history $\mathcal{F}_k$, drawing a random direction sample $\mathbf{z}_{i,k}$ and querying function values from the $\mathcal{SZO}$ are performed in a statistically independent manner. 

\noindent We remark that by Assumption \ref{as:a4},
{\small\begin{align}
\label{eq:0noise_2}
&\mathbb{E}\left[v_{i}(k;\mathbf{x}_{i}(k))\mathbf{z}_{i,k}|\mathcal{F}_{k}\right] = \mathbb{E}\left[\mathbf{z}_{i,k}\mathbb{E}\left[v_{i}(k;\mathbf{x}_{i}(k))|\mathcal{F}_{k},\mathbf{z}_{i,k}\right]\mid\mathcal{F}_{k}\right]\nonumber\\&\Rightarrow \mathbb{E}\left[\mathbf{v}_{\mathbf{z}}(k;\mathbf{x}(k))\mid\mathcal{F}_{k}\right] = \mathbf{0}.
\end{align}}
and,
{\small\begin{align}
\label{eq:0noise_3}
&\mathbb{E}\left[\left\|v_{i}(k;\mathbf{x}_{i}(k))\mathbf{z}_{i,k}\right\|^{2}|\mathcal{F}_{k}\right] \nonumber\\&= \mathbb{E}\left[\left\|\mathbf{z}_{i,k}\right\|^{2}\mathbb{E}\left[v_{i}^{2}(k;\mathbf{x}_{i}(k))|\mathcal{F}_{k},\mathbf{z}_{i,k}\right]|\mathcal{F}_{k}\right] \nonumber\\&\leq \mathbb{E}\left[\left\|\mathbf{z}_{i,k}\right\|^{2}\right]\left(c_{v}\|\mathbf{x}_i(k)\|^2+\sigma_{v}^{2}\right),
\end{align}}
\noindent where if $\mathbf{z}_{i,k}$'s are sampled from a normal distribution with $\mathbf{E}\left[\mathbf{z}_{i,k}\mathbf{z}_{i,k}^{\top}\right] = \mathbf{I}_{d}$ or uniformly from the surface of the $l_2$-ball of radius $\sqrt{d}$, then we have,
{\small\begin{align}
\label{eq:0noise_4}
\mathbb{E}\left[\left\|v_{i}(k;\mathbf{x}_{i}(k))\mathbf{z}_{i,k}\right\|^{2}|\mathcal{F}_{k}\right] \le d\left(c_{v}\|\mathbf{x}_i(k)\|^2+\sigma_{v}^{2}\right).
\end{align}}
\subsection{First Order Optimization}
\label{sec:first_order}
Each node~$i$, $i=1,...,N$, in the network maintains its own optimizer $\mathbf{x}_i(k) \in {\mathbb R}^d$ at each time step~(iterations) $k=0,1,...,$.
\noindent Specifically, for
arbitrary deterministic initial points
$\mathbf{x}_i(0) \in {\mathbb R}^d$, $i=1,...,N$,
the update rule at node $i$
and $k=0,1,...,$
is as follows:
{\small\begin{eqnarray}
\label{eqn-alg-node-i}
\mathbf{x}_i(k+1) &=&
\mathbf{x}_i(k) - \,
\sum_{j \in \Omega_i}
\psi_{i,k}\psi_{j,k}\left( \mathbf{x}_i(k) - \mathbf{x}_j(k) \right)\\
&-&
\alpha_k\,\left( \,\nabla f_i(\mathbf{x}_i(k)) + \mathbf{u}_i(k)\,\right).\nonumber
\end{eqnarray}}
In comparison to the generalized update in \eqref{eq:opt_strategy}, the weights assigned to incoming messages is given by $\gamma_{i,j}(k)=\psi_{i,k}\psi_{j,k}$, while the approximated gradient is given by $\nabla f_i(\mathbf{x}_i(k)) + \mathbf{u}_i(k)$.
The update \eqref{eqn-alg-node-i} is realized in a parallel fashion at all nodes $i=1,...,N$. First, each node $i$, when activated, i.e., when $\psi_{i,k}\neq 0$, broadcasts $\mathbf{x}_i(k)$ to
all its active neighbors $j \in \Omega_i$ which satisfy $\psi_{j,k}\neq 0$ and receives $\mathbf{x}_j(k)$ from all $j \in \Omega_i$ which are active. Subsequently, each node $i$, $i=1,...,N$
makes update~\eqref{eqn-alg-node-i}, which completes an iteration. Finally, $\mathbf{u}_i(k)$ is noise in the calculation of the $f_i$'s gradient at iteration~$k$.
\noindent For $k=0,1,...$, algorithm \eqref{eqn-alg-node-i} can be compactly written as follows:
{\small\begin{align}
\label{eq:update_rule}
\mathbf{x}(k+1) = \mathbf{W}_{k}\mathbf{x}(k)-\alpha_{k}\left(\nabla F(\mathbf{x}(k))+\mathbf{u}(k)\right),
\end{align}}
\noindent where $\mathbf{x} = \left[\mathbf{x}_{1}^{\top},\cdots,\mathbf{x}_{N}^{\top}\right]^{\top}\in\mathbb{R}^{Nd}$ and $\mathbf{u}(k)=\left[\mathbf{u}_{1}^{\top}(k),\cdots,\mathbf{u}_{N}^{\top}(k)\right]^{\top}\in\mathbb{R}^{Nd}$.
\noindent We make the following standard
assumption on the gradient noises.
First, denote by $\mathcal{S}_k$
the history
of algorithm \eqref{eqn-alg-node-i}
up to time $k$;
that is, $\mathcal{S}_k$, $k=1,2,...,$
is an increasing sequence of $\sigma$-algebras,
where
$\mathcal{S}_k$
is the $\sigma$-algebra generated
by the collection of random variables $
\{ \,\boldsymbol{\mathbf{R}}(s),
\,\mathbf{u}_i(t)\}$,
$i=1,...,N$,
$s=0,...,k-1$,
$t=0,...,k-1$.
\begin{myassump}{B2}
	\label{assumption-gradient-noise}
	For each $i=1,...,N$,
	the sequence of noises $\{\mathbf{u}_i(k)\}$
	satisfies for all $k=0,1,...$:
	{\small\begin{eqnarray}
	\label{eqn-ass-noise-1}
	\mathbb{E}[\,\mathbf{u}_i(k)\,|\,\mathcal{S}_k\,] &=& 0,\,\,\mathrm{almost\,surely\,(a.s.)}\\
	\label{eqn-ass-noise-2}
	\mathbb{E}[\,\|\mathbf{u}_i(k)\|^2\,|\,\mathcal{S}_k\,] &\leq& c_{u}\|\mathbf{x}_i(k)\|^2+\sigma_{u}^{2},
	\,\,\mathrm{a.s.},
	\end{eqnarray}}
	where $c_u$ is a nonnegative constant. 
\end{myassump}
\noindent\textbf{Communication Cost}
Define the communication cost $\mathcal{C}_{k}$ to be the expected per-node number of transmissions up to iteration $k$, i.e.,
{\small\begin{align}
\label{eq:comm_cost_1}
\mathcal{C}_{k}= \mathbb{E}\left[\sum_{s=0}^{k-1}\mathbb{I}_{\{\textit{node}~C~\textit{transmits~at}~s\}}\right],
\end{align}}
\noindent where $\mathbb{I}_{A}$ represents the indicator of event $A$. Note that the per-node communication cost in \eqref{eq:comm_cost_1} is the same as the network average of communication costs across all nodes, as the activation probabilities are homogeneous across nodes.
We now proceed to the main results pertaining to the proposed zeroth order and first order optimization schemes.
\section{Convergence rates: Statement of main results and interpretations}
\label{sec:main_results}
\noindent In this section, we state the results for both the zeroth order and the first order optimization algorithms. 
\subsection{Main Results: Zeroth Order Optimization}
We state the main result concerning the mean square error at each agent $i$ next.
\begin{Theorem}
	\label{theorem-1}
	1) Consider the optimizer estimate sequence $\{\mathbf{x}(k)\}$ generated by the algorithm \eqref{eq:update_rule_node}. Let assumptions \ref{as:1}-\ref{as:3} and \ref{as:a1}-\ref{as:a4} hold. Then, for each node~$i$'s optimizer estimate $\mathbf{x}_i(k)$ and the solution $\mathbf{x}^\star$ of problem~\eqref{eq:opt_problem}, $\forall k \geq 0$ there holds:
	{\small\begin{align}
	\label{eq:th1}
	&\mathbb{E}\left[\left\|\mathbf{x}_{i}(k)-\mathbf{x}^{\ast}\right\|^{2}\right] \le 2M_{k}+\frac{64NL^{2}\Delta_{1,\infty}\alpha_{0}^{2}}{\mu^2\lambda_{2}^{2}\left(\overline{\mathbf{R}}\right)c_0^{2}\beta_0^{2}(k+1)^{2-2\tau-2\delta}}\nonumber\\&\frac{16NM^{2}d^2(P)c_{0}^{4}}{\mu^{2}(k+1)^{4\delta}}+2Q_{k}+
	\frac{8\Delta_{1,\infty}\alpha_{0}^{2}}{\lambda_{2}^{2}\left(\overline{\mathbf{R}}\right)\beta_0^{2}c_0^2(k+1)^{2-2\tau-2\delta}}\nonumber\\&+
	\frac{4N\alpha_{0}\left(dc_{v}q_{\infty}(N,d,\alpha_0,c_0)+dN\sigma_1^{2}\right)}{\mu c_0^2(k+1)^{1-2\delta}},
	\end{align}}
	where, {\small$\Delta_{1,\infty}=6dc_{v}q_{\infty}(N,d,\alpha_0,c_0)+6dN\sigma_1^{2}$}  and {\small$q_{\infty}(N,d,\alpha_0,c_0)=\mathbb{E}\left[\left\|\mathbf{x}(k_2)-\mathbf{x}^{o}\right\|^{2}\right]+4\frac{\|\nabla F(\mathbf{x}^{o})\|^2}{\mu^2}+\frac{\sqrt{N}s_{1}(P)M\alpha_0c_0^2}{8\delta}+\frac{Ns_{1}^{2}(P)M^{2}\alpha_0^2c_0^4}{16(1+4\delta)}+\frac{d\alpha_0^2\left(2c_{v}N\left\|\mathbf{x}^{o}\right\|^{2}+N\sigma_{v}^{2}\right)}{c_0^{2}(1-2\delta)}+\frac{\alpha_0^2c_0^2\sqrt{N}s_{1}(P)M\left\|\nabla F(\mathbf{x}^{o})\right\|}{1+2\delta}+\frac{2N\alpha_{0}^{2}c_{0}^{4}s_{2}(P)}{1+4\delta}+\frac{4\alpha_{0}^{2}c_{0}^{2}Ns_{1}(P)}{1+2\delta}\left\|\nabla F\left(\mathbf{x}^{o}\right)\right\|^{2}$, $k_2=\max\{k_0,k_1\}$, $k_0 = \inf\{k|\mu^{2}\alpha_{k}^2 < 1\}$} and {\small$k_{1}=\inf\left\{k | \frac{\mu}{2} > \frac{\sqrt{N}}{4}s_{1}(P)Mc_{k}^2+\frac{2dc_{v}\alpha_{k}}{c_{k}^{2}}+4\alpha_{k}c_{k}^{2}Ns_{1}(P)L^{2}\right\}$}, with $M_{k}$ and $Q_{k}$ decaying faster than the rest of the terms.\\
	2) In particular, the rate of decay of the RHS of \eqref{eq:th1} is given by $(k+1)^{-\delta_1}$, where $\delta_{1}=\min\left\{1-2\delta,2-2\tau-2\delta,4\delta\right\}$. By, optimizing over $\tau$ and $\delta$, we obtain that for $\tau=1/2$ and $\delta=1/6$,
	{\small\begin{align*}
	&\mathbb{E}\left[\left\|\mathbf{x}_{i}(k)-\mathbf{x}^{\ast}\right\|^{2}\right]
	\le 2M_{k}+\frac{32NL^{2}\Delta_{1,\infty}\alpha_{0}^{2}}{\mu^2\lambda_{2}^{2}\left(\overline{\mathbf{R}}\right)c_0^{2}\beta_0^{2}(k+1)^{2/3}}\nonumber\\&\frac{16NM^{2}d^2(P)c_{0}^{2}}{\mu^{2}(k+1)^{2/3}}+2Q_{k}+
	\frac{8\Delta_{1,\infty}\alpha_{0}^{2}}{\lambda_{2}^{2}\left(\overline{\mathbf{R}}\right)\beta_0^{2}c_0^2(k+1)^{2/3}}\nonumber\\&+
	\frac{4N\alpha_{0}\left(dc_{v}q_{\infty}(N,d,\alpha_0,c_0)+dN\sigma_1^{2}\right)}{\mu c_0^2(k+1)^{2/3}}= O\left(\frac{1}{k^{\frac{2}{3}}}\right),~~\forall i.
	\end{align*}}\\
	3) The communication cost is given by,
	{\small\begin{align*}
	%\label{eq:comm_cost1}
	\mathbb{E}\left[\sum_{t=1}^{k}\zeta_{t}\right]=O\left(k^{\frac{3}{4}+\frac{\epsilon}{2}}\right).
	\end{align*}}
	and the MSE-communication rate is given by,
	{\small\begin{align}
	\label{eq:th0comm_rate}
	\mathbb{E}\left[\left\|\mathbf{x}_{i}(k)-{\mathbf{x}^{\star}}\right\|^{2}\right]=O\left(\mathcal{C}_{k}^{-8/9+\zeta}\right),
	\end{align}}
	where $\zeta$ can be arbitrarily small.
\end{Theorem}
\noindent Theorem \ref{theorem-1} asserts an $O\left(\mathcal{C}_{k}^{-8/9+\zeta}\right)$ MSE-communication rate can be achieved while keeping the MSE decay rate at $O\left(k^{-\frac{2}{3}}\right)$. The performance of the zeroth order optimization scheme depends explicitly on the connectivity of the expected Laplacian through the terms $\frac{32NL^{2}\Delta_{1,\infty}\alpha_{0}^{2}}{\mu^2\lambda_{2}^{2}\left(\overline{\mathbf{R}}\right)c_0^{2}\beta_0^{2}(k+1)^{0.5}}$ and $\frac{8\Delta_{1,\infty}\alpha_{0}^{2}}{\lambda_{2}^{2}\left(\overline{\mathbf{R}}\right)\beta_0^{2}c_0^2(k+1)^{0.5}}$. In particular, communication graphs which are well connected, i.e., have higher values of $\lambda_{2}\left(\overline{\mathbf{R}}\right)$ will have lower MSE as compared to a counterpart with lower values of $\lambda_{2}\left(\overline{\mathbf{R}}\right)$.\\
If higher order smoothness assumptions are made, i.e., a $p$-th order smoothness assumption is made which is then exploited by means of a $p$-th degree finite difference gradient approximation, then by repeating the same proof arguments, the rate in terms of iteration count can be shown to improve to $O\left(k^{-\frac{p}{p+1}}\right)$. The improvement can be attributed to a better bias-variance tradeoff as illustrated by the terms $\frac{8M^{2}d^2(P)c_{0}^{4}}{\mu^{2}(k+1)^{2p\delta}}$ and $\frac{4N\alpha_{0}\left(dc_{v}q_{\infty}(N,d,\alpha_0,c_0)+dN\sigma_1^{2}\right)}{\mu c_0^2(k+1)^{1-2\delta}}$. The corresponding MSE-communication rate improves to $O\left(\mathcal{C}_{k}^{-\frac{4p}{3(p+1)}+\zeta}\right)$.
\subsection{Main Results: First Order Optimization}
We state the main result concerning the mean square error at each agent $i$ next.
\begin{Theorem}
	\label{theorem-2}
	Consider algorithm \eqref{eqn-alg-node-i} with
	step-sizes $\alpha_k=\frac{\alpha_0}{k+1}$
	and $\beta_k=\frac{\beta_0}{(k+1)^{1/2}}$,
	where $\beta_0>0$ and
	$\alpha_0>2/\mu$. Further, let Assumptions \ref{as:1}-\ref{as:3} and \ref{assumption-gradient-noise} hold.\\
	1) Then, for each node~$i$'s
	solution estimate $\mathbf{x}_i(k)$ and
	the solution $\mathbf{x}^\star$ of problem~\eqref{eq:opt_problem}, , $\forall k \geq 0$ there holds:
	{\small\begin{align}
	\label{eq:th2}
	&\mathbb{E}\left[\left\|\mathbf{x}_{i}(k)-\mathbf{x}^{\ast}\right\|^{2}\right] \le 2M_{k}+\frac{32NL^{2}\Delta_{1,\infty}\alpha_{0}^{2}}{\mu^2\lambda_{2}^{2}\left(\overline{\mathbf{R}}\right)\beta_0^{2}(k+1)}\nonumber\\&+2Q_{k}+ \frac{4\Delta_{1,\infty}\alpha_{0}^{2}}{\lambda_{2}^{2}\left(\overline{\mathbf{R}}\right)\beta_0^{2}(k+1)},
	\end{align}}
	where, {\small$\Delta_{1,\infty}=2\left\|\nabla F(\mathbf{x}(k))\right\|^{2}+4c_{u}q_{\infty}(N,\alpha_0)+4N\sigma_{1}^{2}$}  and {\small$q_{\infty}(N,\alpha_0)=\mathbb{E}\left[\left\|\mathbf{x}(k_2)-\mathbf{x}^{o}\right\|^{2}\right]+\frac{\pi^{2}}{6}\alpha_0^2\left(2c_{u}N\left\|\mathbf{x}^{o}\right\|^{2}+N\sigma_{u}^{2}\right)+4\frac{\|\nabla F(\mathbf{x}^{o})\|^2}{\mu^2}$}, {\small$k_2=\max\{k_0,k_1\}$, $k_0 = \inf\{k|\mu^{2}\alpha_{k}^2 < 1\}$} and {\small$k_{1}=\inf\left\{k | \frac{\mu}{2} > 2c_{u}\alpha_{k}\right\}$}, with $M_{k}$ and $Q_{k}$ decaying faster than the rest of the terms.\\
	2) The communication cost is given by,
	{\small\begin{align*}
	%\label{eq:comm_cost1}
	\mathbb{E}\left[\sum_{t=1}^{k}\zeta_{t}\right]=O\left(k^{\frac{3}{4}+\frac{\epsilon}{2}}\right),
	\end{align*}}
	leading to the following MSE-communication rate:
	{\small\begin{align}
	\label{eq:1comm_rate}
	\mathbb{E}\left[\left\|\mathbf{x}_{i}(k)-{\mathbf{x}^{\star}}\right\|^{2}\right]=O\left(\mathcal{C}_{k}^{-\frac{4}{3}+\zeta}\right),
	\end{align}}
	where $\zeta$ can be arbitrarily small.
\end{Theorem}
We remark that
the condition
$\alpha_0>2/\mu$ can be
relaxed to
require only a positive
$\alpha_0$,
in which case the rate becomes
$O(\mathrm{ln}(k)/k)$, instead
of~$O(1/k)$.
Also, to avoid
large step-sizes
at initial iterations for a large $\alpha_0$,
step-size $\alpha_k$
can be modified to
$\alpha_k=\alpha_0/(k+k_0)$,
for arbitrary positive constant~$k_0$,
and Theorem~\ref{theorem-2}
continues to hold.
Theorem~\ref{theorem-2}
establishes the $O(1/k)$ MSE rate of convergence
of algorithm~\eqref{eqn-alg-node-i}; due to the
assumed $f_i$'s strong convexity, the theorem also implies
that $\mathbb{E}\left[ f(\mathbf{x}_i(k)) - f(\mathbf{x}^\star)\right]
=O(1/k)$.
\section{Simulations}
\label{sec:sim}
\noindent In this section, we provide evaluations of the proposed algorithms on the Abalone dataset~(\cite{LibSVM}). To be specific, we consider $\ell_2$-regularized empirical risk minimization for the Abalone dataset, where the regularization function is given by $\Psi_i(\mathbf{x})=
\frac{1}{2}\|\mathbf{x}\|^2$. We consider a $10$ node network for both the zeroth and first order optimization schemes. The Abalone dataset has $4177$ data points out of which $577$ data points are kept aside as the test set and the other $3600$ is divided equally among the $10$ nodes resulting in each node having $360$ data points. For the zeroth order optimization, we compare the proposed undirected sequence of Laplacian constructions based optimization scheme and the static Laplacian~(Benchmark) based optimization schemes. The benchmark scheme  is characterized by the communication graph being static and thereby resulting agents connected through a link to exchange messages at all times. The data points at each node are sampled without replacement in a contiguous manner. The vectors $\mathbf{z}_{i,k}$s for evaluating directional derivatives were sampled from a normal distribution with identity covariance. Figure \ref{Figure_1} compares the test error for the three aforementioned schemes, where it can be clearly observed that the test error is indistinguishable in terms of the number of iterations or equivalently in terms of the number of queries to the stochastic zeroth oracle. Figure \ref{Figure_2} demonstrates the superiority the proposed algorithm in terms of the test error versus communication cost as compared to the benchmark as predicted by Theorem \ref{theorem-1}. For example, at the same relative test error level, the proposed algorithm uses up to $3$x less number of transmissions as compared to the benchmark scheme.
\begin{figure}[thpb]
	\centering
	\includegraphics[width=90mm]{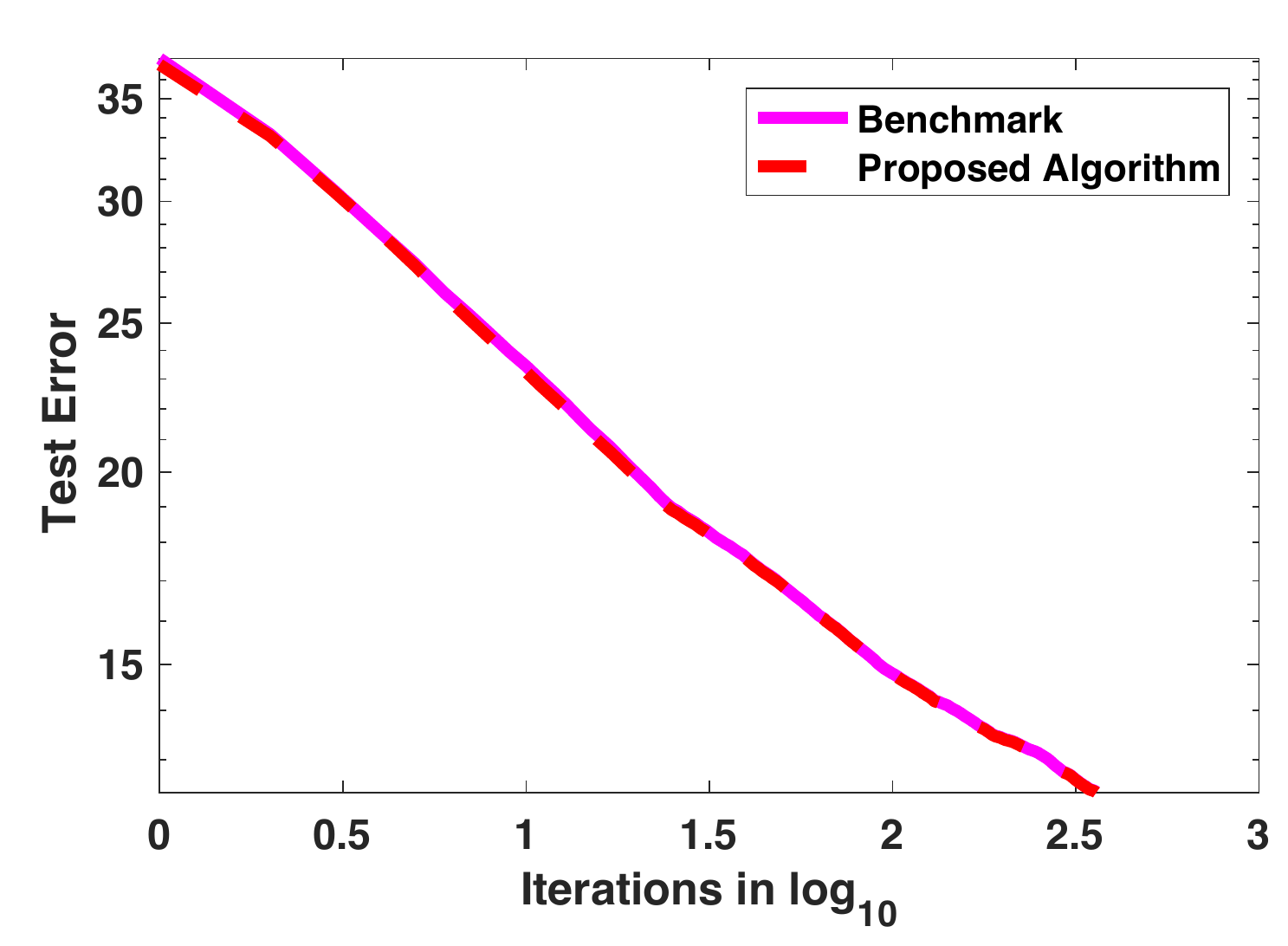}
	\caption{Test Error vs Iterations}
	\label{Figure_1}
\end{figure}
\begin{figure}[thpb]
	\centering
	\includegraphics[width=90mm]{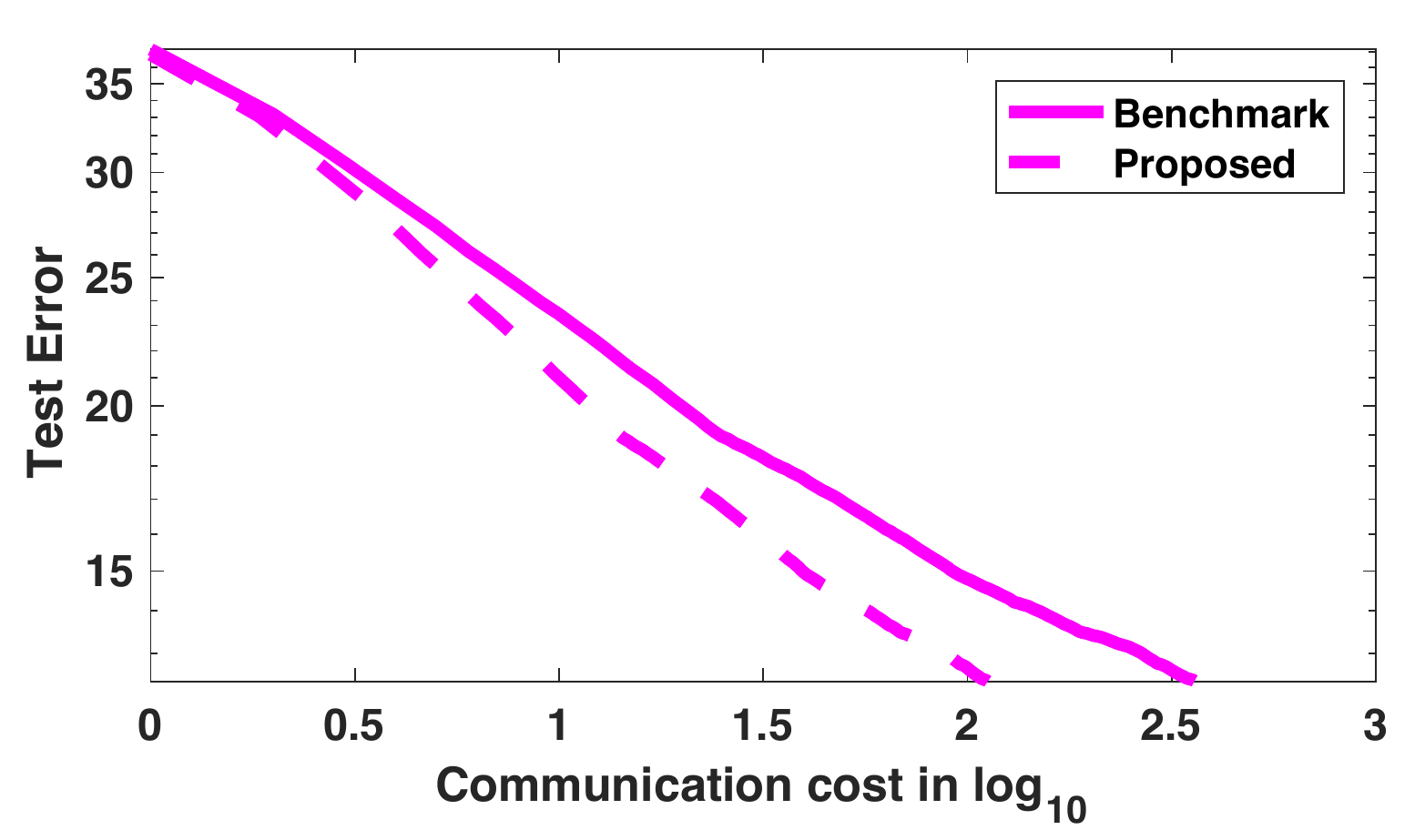}
	\caption{Test Error vs Communication Cost}
	\label{Figure_2}
\end{figure}
In Figure \ref{Figure_3}, the test error of the communication efficient first order optimization scheme is compared with the test error of the benchmark scheme which refers to the optimization scheme with the communication graph abstracted by a static Laplacian in terms of iterations or equivalently the number of queries per agent to the stochastic first order oracle, i.e., gradient evaluations. Figure \ref{Figure_4} demonstrates the superiority of the proposed communication efficient first order optimization scheme in terms of the test error versus communication cost as compared to the benchmark as predicted by Theorem \ref{theorem-2}. For example, at the same relative test error level, the proposed algorithm uses up to $3$x less number of transmissions as compared to the benchmark scheme. 
\begin{figure}[thpb]
	\centering
	\includegraphics[width=90mm]{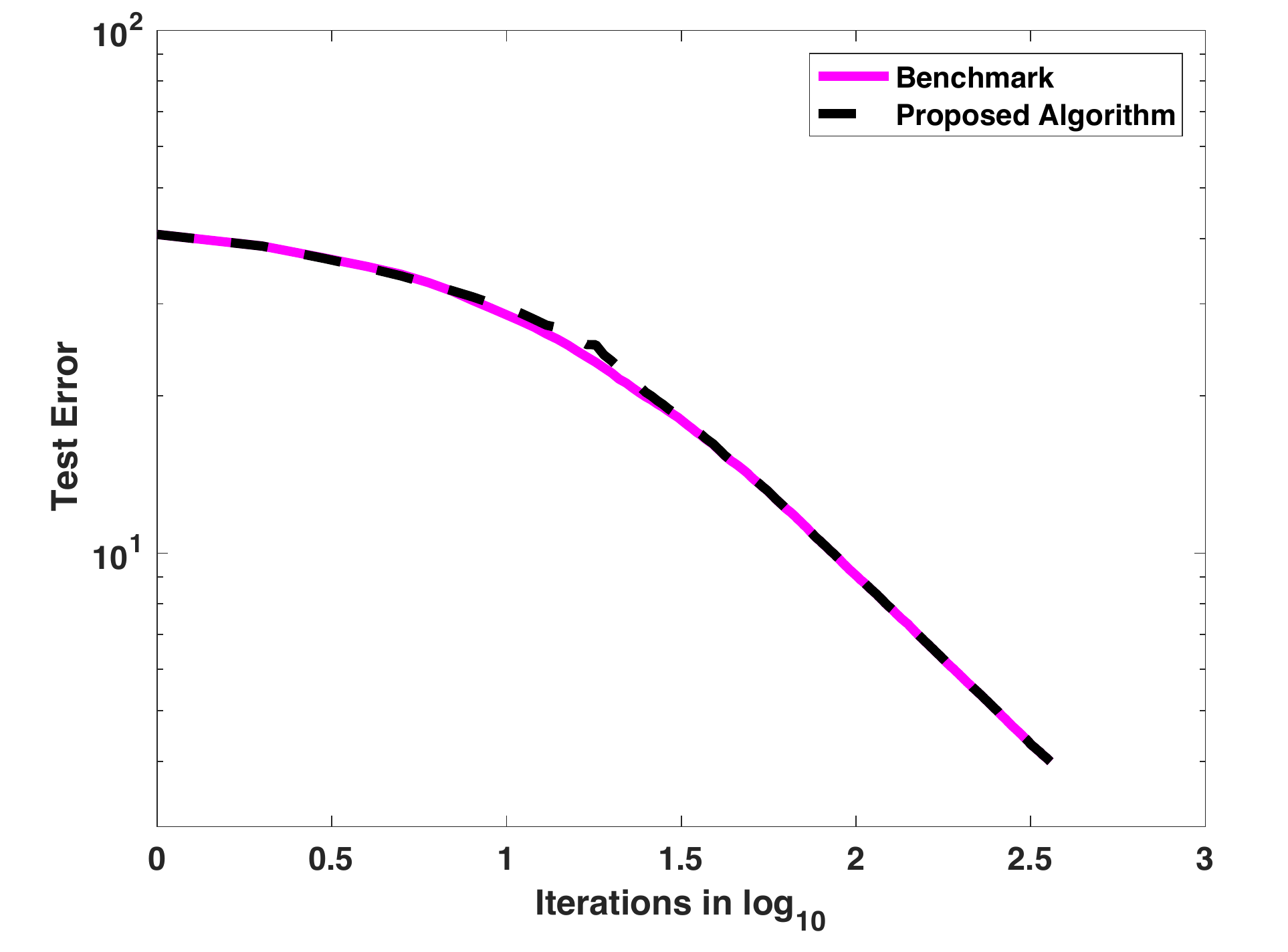}
	\caption{Test Error vs Iteration}
	\label{Figure_3}
\end{figure}
\begin{figure}[thpb]
	\centering
	\includegraphics[width=90mm]{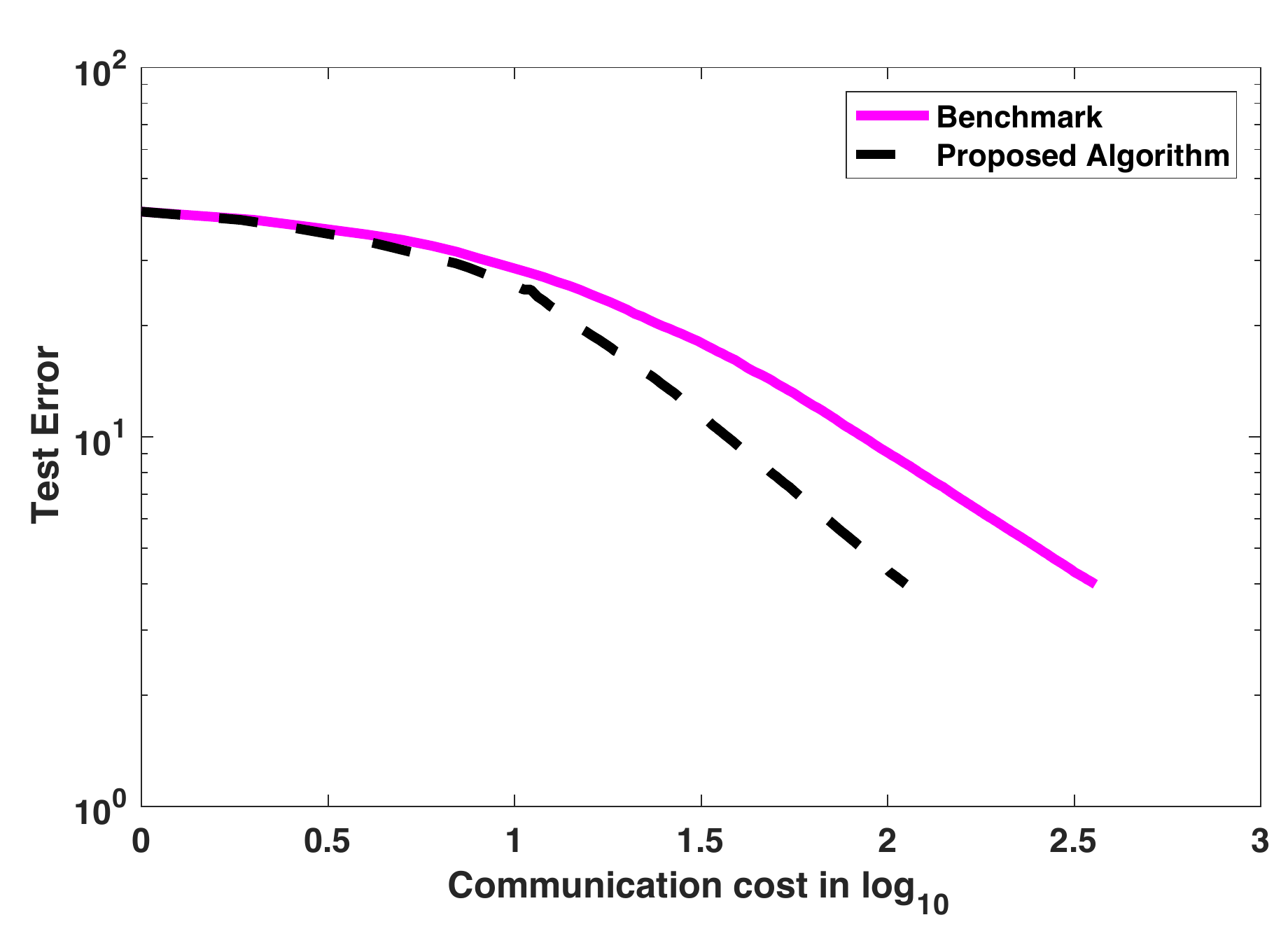}
	\caption{Test Error vs Communication Cost}
	\label{Figure_4}
\end{figure}
\section{Proof of the main result: Zeroth order optimization}
\label{sec:proof_main_res_0}
The proof of the main result proceeds through three main steps. The first step involves establishing the boundedness of the iterate sequence, while the second step involves establishing the convergence rate of the optimizer sequence at each agent to the network averaged optimizer sequence. The convergence of the network averaged optimizer is then analyzed as the final step and in combination with the second step results in the establishment of bounds on MSE of the optimizer sequence at each agent.
\begin{Lemma}
	\label{Lemma0-MSS-BDD}
	Let the hypotheses of Theorem \ref{theorem-1} hold. 
	Then, we have,
	{\small\begin{align*}
	&\mathbb{E}\left[\left\|\mathbf{x}(k)-\mathbf{x}^{o}\right\|^2\right]\le q_{k_2}(N,d,\alpha_0,c_0)+4\frac{\|\nabla F(\mathbf{x}^{o})\|^2}{\mu^2}\nonumber\\&+\frac{\sqrt{N}s_{1}(P)M\alpha_0c_0^2}{8\delta}+\frac{Ns_{1}^{2}(P)M^{2}\alpha_0^2c_0^4}{16(1+4\delta)}\nonumber\\&+\frac{d\alpha_0^2\left(2c_{v}N\left\|\mathbf{x}^{o}\right\|^{2}+N\sigma_{v}^{2}\right)}{c_0^{2}(1-2\delta)}+\frac{\alpha_0^2c_0^2\sqrt{N}s_{1}(P)L\left\|\nabla F(\mathbf{x}^{o})\right\|}{1+2\delta}\nonumber\\&+\frac{N\alpha_{0}^{2}c_{0}^{4}s_{2}(P)}{1+4\delta}+\frac{4\alpha_{0}^{2}c_{0}^{2}Ns_{1}(P)}{1+2\delta}\left\|\nabla F\left(\mathbf{x}^{o}\right)\right\|^{2}\nonumber\\&\doteq q_{\infty}(N,d,\alpha_0,c_0),
	\end{align*}}
	where {\small$\mathbb{E}\left[\left\|\mathbf{x}(k_2)-\mathbf{x}^{o}\right\|^{2}\right] \le q_{k_2}(N,d,\alpha_0,c_0)$, $k_2=\max\{k_0,k_1\}$, $k_0 = \inf\{k|\mu^{2}\alpha_{k}^2 < 1\}$} and {\small$k_{1}=\inf\left\{k | \frac{\mu}{2} > \frac{\sqrt{N}}{4}s_{1}(P)Mc_{k}^2+\frac{2dc_{v}\alpha_{k}}{c_{k}^{2}}+4\alpha_{k}c_{k}^{2}Ns_{1}(P)L^{2}\right\}$}.
\end{Lemma}
\begin{proof}
	{\small\begin{align}
	\label{eq:update_rule1}
	&\mathbf{x}(k+1) = \mathbf{W}_{k}\mathbf{x}(k)\nonumber\\&-\frac{\alpha_{k}}{c_{k}}\left(c_{k}\nabla F(\mathbf{x}(k))+c_{k}^{2}\mathbf{b}(\mathbf{x}(k))+c_{k}\mathbf{h}(\mathbf{x}(k))\right).
	\end{align}}
	Denote $\mathbf{x}^{o} = \mathbf{1}_{N}\otimes x^{\ast}$.
	Then, we have,
	{\small\begin{align}
	\label{eq:0update_rule2}
	&\mathbf{x}(k+1)-\mathbf{x}^{o} = \mathbf{W}_{k}(\mathbf{x}(k)-\mathbf{x}^{o})\nonumber\\&-\alpha_{k}\left(\nabla F(\mathbf{x}(k))-\nabla F(\mathbf{x}^{o})\right)\nonumber\\&-\alpha_{k}\mathbf{h}(\mathbf{x}(k))-\alpha_{k}\nabla F(\mathbf{x}^{o})-\alpha_{k}c_{k}\mathbf{b}(\mathbf{x}(k)).
	\end{align}}
	Moreover, note that, $\mathbb{E}\left[\mathbf{h}(\mathbf{x}(k))\mid\mathcal{F}_{k}\right]=0$.
	By Leibnitz rule, we have,
	{\small\begin{align}
	\label{eq:0mvt}
	&\nabla F(\mathbf{x}(k))-\nabla F(\mathbf{x}^{o}) \nonumber\\&= \left[\int_{s=0}^{1}\nabla^{2}F\left(\mathbf{x}^{o}+s(\mathbf{x}(k)-\mathbf{x}^{o})\right)ds\right]\left(\mathbf{x}(k)-\mathbf{x}^{o}\right)\nonumber\\
	&=\mathbf{H}_{k}\left(\mathbf{x}(k)-\mathbf{x}^{o}\right).
	\end{align}}
	By Lipschitz continuity of the gradients and strong convexity of $f(\cdot)$, we have that $L\mathbf{I}\succcurlyeq\mathbf{H}_{k}\succcurlyeq\mu\mathbf{I}$. 
	Denote by $\boldsymbol{\zeta}(k) = \mathbf{x}(k)-\mathbf{x}^{o}$
	and by $\boldsymbol{\xi}(k) = \left(\mathbf{W}_{k}-\alpha_{k}\mathbf{H}_{k}\right)(\mathbf{x}(k)-\mathbf{x}^{o})
	-\alpha_k \,\nabla F(\mathbf{x}^{o})$.
	Then, there holds:
	{\small\begin{align}
		&\mathbb{E}[\,\|\boldsymbol{\zeta}(k+1)\|^2 \,|\,\mathcal{F}_k\,]
		\leq
		\mathbb{E}\left[\|\boldsymbol{\xi}(k)\|^2|\mathcal{F}_{k}\right] \nonumber \\
		&-2 \alpha_k c_{k}\,\mathbb{E}\left[{\boldsymbol{\xi}}(k)^\top|\,\mathcal{F}_k\,\right]
		\mathbb{E}[\,\mathbf{h}(\mathbf{x}(k)) \,|\,\mathcal{F}_k\,] +
		\alpha_k^2c_{k}^{2} \,\mathbb{E}[\,\|\mathbf{h}(\mathbf{x}(k))\|^2 \,|\,\mathcal{F}_k\,] \nonumber \\
		&+\alpha_{k}^{2}c_{k}^{2}\mathbf{b}^{\top}(\mathbf{x}(k))\mathbf{b}(\mathbf{x}(k))-2\alpha_{k}c_{k}\mathbf{b}^{\top}(\mathbf{x}(k))\mathbb{E}\left[{\boldsymbol{\xi}}(k)|\mathcal{F}_{k}\right]\nonumber\\
		&+\mathbf{b}\left(\mathbf{x}(k)\right)^\top \mathbb{E}\left[\mathbf{h}(\mathbf{x}(k))|\mathcal{F}_k\right].
		\label{eqn-0combine-1}
		\end{align}}
	We use the following inequalities:
	{\small\begin{align}
		\label{eq:bias1}
		&c_{k}\mathbf{b}(\mathbf{x}_{i}(k))\nonumber\\& = \frac{c_{k}}{2}\mathbb{E}\left[\langle\mathbf{z}_{i,k},\nabla^{2}f_{i}\left(\mathbf{x}_{i}(k)+\frac{\left(1-\theta_1\right)}{2}c_{k}\mathbf{z}_{i,k}\right)\mathbf{z}_{i,k}\rangle\mathbf{z}_{i,k}|\mathcal{F}_{k}\right]\nonumber\\&-\frac{c_{k}}{2}\mathbb{E}\left[\langle\mathbf{z}_{i,k},\nabla^{2}f_{i}\left(\mathbf{x}_{i}(k)+\left(1-\theta_2\right)c_{k}\mathbf{z}_{i,k}\right)\mathbf{z}_{i,k}\rangle\mathbf{z}_{i,k}|\mathcal{F}_{k}\right]\nonumber\\
		&\Rightarrow c_{k}\left\|\mathbf{b}(\mathbf{x}_{i}(k))\right\|\le \frac{c_{k}^2}{4}Ms_{1}(P).
		\end{align}}
	{\small\begin{align}
		\label{eq:cross_term}
		&-\mathbf{b}^{\top}(\mathbf{x}(k))\mathbb{E}\left[{\boldsymbol{\xi}}(k)|\mathcal{F}_{k}\right]\nonumber\\
		&=-2\mathbf{b}^{\top}(\mathbf{x}(k))\left(\mathbf{I}-\beta_{k}\overline{\mathbf{R}}-\alpha_{k}\mathbf{H}_{k}\right)(\mathbf{x}(k)-\mathbf{x}^{o})\nonumber\\&+2\alpha_{k}\mathbf{b}^{\top}(\mathbf{x}(k))\nabla F(\mathbf{x}^{o})\nonumber\\&\le 2\left\|\mathbf{b}(\mathbf{x}(k))\right\|\left\|\mathbf{I}-\beta_{k}\overline{\mathbf{R}}-\alpha_{k}\mathbf{H}_{k}\right\|\left\|\mathbf{x}(k)-\mathbf{x}^{o}\right\|\nonumber\\&+2\alpha_{k}\left\|\mathbf{b}(\mathbf{x}(k))\right\|\left\|\nabla F(\mathbf{x}^{o})\right\|\nonumber\\
		&\le \frac{\sqrt{N}}{4}s_{1}(P)Mc_{k}\left(1-\mu\alpha_{k}\right)\left(1+\left\|\mathbf{x}(k)-\mathbf{x}^{o}\right\|^{2}\right)\nonumber\\&+\alpha_{k}c_{k}\frac{\sqrt{N}}{2}s_{1}(P)M\left\|\nabla F(\mathbf{x}^{o})\right\|\nonumber\\
		&\le \frac{\sqrt{N}}{4}s_{1}(P)Mc_{k}+\frac{\sqrt{N}}{4}s_{1}(P)Mc_{k}\left\|\mathbf{x}(k)-\mathbf{x}^{o}\right\|^{2}\nonumber\\&+\alpha_{k}c_{k}\frac{\sqrt{N}}{2}s_{1}(P)M\left\|\nabla F(\mathbf{x}^{o})\right\|,
		\end{align}}
	{\small\begin{align}
		\label{eq:bias}
		\mathbf{b}^{\top}(\mathbf{x}(k))\mathbf{b}(\mathbf{x}(k)) \le \frac{N}{16}s_{1}^{2}(P)M^{2}c_{k}^{2},
		\end{align}}
	{\small\begin{align}
		\label{eq:error}
		&\,\mathbb{E}[\,\|\mathbf{h}(\mathbf{x}(k))\|^2 \,|\,\mathcal{F}_k\,] = \mathbb{E}\left[\left\|\mathbf{v}_{\mathbf{z}}(k;\mathbf{x}(k))\right\|^{2}|\mathcal{F}_{k}\right]\nonumber\\&+\mathbb{E}\left[\left\|\mathbf{g}(\mathbf{x}(k))-\mathbb{E}\left[\widehat{\mathbf{g}}(\mathbf{x}(k))\mid\mathcal{F}_{k}\right]\right\|^{2}\mid\mathcal{F}_{k}\right],
		\end{align}}
	{\small\begin{align}
		\label{eq:error1}
		&\mathbb{E}\left[\left\|\mathbf{g}(\mathbf{x}(k))-\mathbb{E}\left[\widehat{\mathbf{g}}(\mathbf{x}(k))\mid\mathcal{F}_{k}\right]\right\|^{2}\mid\mathcal{F}_{k}\right]\nonumber\\
		&\le\mathbb{E}\left[\left\|\mathbf{g}(\mathbf{x}(k))\right\|^{2}\mid\mathcal{F}_{k}\right]\nonumber\\&\le 4Ns_{1}(P)L^{2}\left\|\mathbf{x}(k)-\mathbf{x}^{o}\right\|^{2}+4Ns_{1}(P)\left\|\nabla F\left(\mathbf{x}^{o}\right)\right\|^{2}+2Nc_{k}^{2}s_{2}(P),
		\end{align}}
	and 
	{\small\begin{align}
		\label{eq:variance0}
		&\mathbb{E}\left[\left\|\mathbf{v}_{\mathbf{z}}(k;\mathbf{x}(k))\right\|^{2}|\mathcal{F}_{k}\right]\le dc_{v}\left\|\mathbf{x}(k)\right\|^{2}+dN\sigma_{v}^{2}\nonumber\\
		&\le 2dc_{v}\left\|\mathbf{x}(k)-\mathbf{x}^{o}\right\|^{2}+\left(2dc_{v}\left\|\mathbf{x}^{o}\right\|^{2}+N\sigma_{v}^{2}\right).
		\end{align}}
	Then from \eqref{eqn-0combine-1}, we have,
	{\small\begin{align}
	\label{eq:eqn-combine-2}
	&\mathbb{E}[\,\|\boldsymbol{\zeta}(k+1)\|^2 \,|\,\mathcal{F}_k\,]\le \mathbb{E}\left[\|\boldsymbol{\xi}(k)\|^2|\mathcal{F}_{k}\right]\nonumber\\
	&+\frac{\sqrt{N}}{4}s_{1}(P)M\alpha_{k}c_{k}^2\|\boldsymbol{\zeta}(k)\|^2+ 2\frac{d\alpha_{k}^{2}}{c_{k}^{2}}c_{v}\|\boldsymbol{\zeta}(k)\|^2\nonumber\\
	&+\frac{d\alpha_{k}^{2}}{c_{k}^{2}}\left(2c_{v}\left\|\mathbf{x}^{o}\right\|^{2}+N\sigma_{v}^{2}\right)+\frac{\sqrt{N}}{4}s_{1}(P)M\alpha_{k}c_{k}^2\nonumber\\
	&+\frac{N}{16}s_{1}^{2}(P)M^{2}\alpha_{k}^{2}c_{k}^{4}+\alpha_{k}^{2}c_{k}^2\frac{\sqrt{N}}{2}s_{1}(P)M\left\|\nabla F(\mathbf{x}^{o})\right\|\nonumber\\
	&+4\alpha_{k}^{2}c_{k}^{2}Ns_{1}(P)L^{2}\|\boldsymbol{\zeta}(k)\|^2+4\alpha_{k}^{2}c_{k}^{2}Ns_{1}(P)\left\|\nabla F\left(\mathbf{x}^{o}\right)\right\|^{2}\nonumber\\&+2N\alpha_{k}^{2}c_{k}^{4}s_{2}(P).
	\end{align}}
	
	\noindent We next bound $\mathbb{E}\left[\|\boldsymbol{\xi}(k)\|^2|\,\mathcal{F}_k\,\right]$.
	Note that
	$\|\mathbf{W}_k -\alpha_k \,\boldsymbol{H}_k\| \leq 1-\mu\,\alpha_k$.
	Therefore, we have:
	\begin{equation}
	\label{0eqn-xi-zeta}
	\|\boldsymbol{\xi}(k)\| \leq (1-\mu\,\alpha_k)\,\|\boldsymbol{\zeta}(k)\|
	+ \alpha_k\, \|\nabla F(\mathbf{x}^{o})\|.
	\end{equation}
	
	%	We next bound $\mathbb{E}\left[\|\boldsymbol{\xi}(k)\|^2|\,\mathcal{F}_k\,\right]$.
	%	%
	%	{\small\begin{align}
	%	\label{eqn0-xi-zeta}
	%	&\mathbb{E}\left[\left\|\left(\mathbf{W}_{k}-\alpha_{k}\mathbf{H}_{k}\right)(\mathbf{x}(k)-\mathbf{x}^{o})\right\|^{2}|\mathcal{F}_{k}\right]\nonumber\\
	%	&=\boldsymbol{\zeta}(k)^{\top}\mathbb{E}\left[\left(\mathbf{W}_{k}-\alpha_{k}\mathbf{H}_{k}\right)^{\top}\left(\mathbf{W}_{k}-\alpha_{k}\mathbf{H}_{k}\right)|\mathcal{F}_{k}\right]\boldsymbol{\zeta}(k)\nonumber\\
	%	&=\boldsymbol{\zeta}(k)^{\top}\mathbb{E}\left[\mathbf{I}-2\beta_{k}\overline{\mathbf{R}}\otimes\mathbf{I}_{d}+\alpha_{k}\beta_{k}\mathbf{H}_{k}\left(\overline{\mathbf{R}}\otimes\mathbf{I}_{d}\right)+\alpha_{k}\beta_{k}\left(\overline{\mathbf{R}}\otimes\mathbf{I}_{d}\right)\mathbf{H}_{k}\right.\nonumber\\&\left.-2\alpha_{k}\mathbf{H}_{k}+\beta_{k}^{2}\overline{\mathbf{R}}^{2}\otimes\mathbf{I}_{d}+\widetilde{\mathbf{R}}^{\top}(k)\widetilde{\mathbf{R}}(k)\otimes\mathbf{I}_{d}+\alpha_{k}^{2}\mathbf{H}_{k}^{2}\right]\boldsymbol{\zeta}(k)
	%	\le (1-2\mu\alpha_{k})\left\|\boldsymbol{\zeta}(k)\right\|^{2},
	%	\end{align}}
	%	for $k\geq k_0$, where $k_{0}=\inf\{k|2\beta_{k}\lambda_{2}\left(\overline{\mathbf{R}}\right)-2\alpha_{k}\beta_{k}L\lambda_{N}\left(\overline{\mathbf{R}}\right)-\beta_{k}^{2}\lambda_{N}^{2}\left(\overline{\mathbf{R}}\right)-\alpha_{k}^{2}L^{2}-4N^{2}\beta_{k}\rho_{k} > 0\}$. Note that, such a $k_0$ always exists as $\alpha_{k},\beta_k$ and $\rho_k$ are time-decaying sequences.
	\noindent We now
	use the following inequality:
	{\small\begin{align}
	\label{eq:cool_ineq_1}
	(a+b)^{2} \leq \left(1+\theta\right)a^{2}+\left(1+\frac{1}{\theta}\right)b^{2},
	\end{align}}
	for any $a,b \in \mathbb{R}$ and $\theta > 0$.
	We set $\theta=\mu\alpha_{k}$.
	Using the inequality \eqref{eq:cool_ineq_1} in \eqref{0eqn-xi-zeta} and we have $\forall k \geq k_0$, where $k_0 = \inf\{k|\mu^{2}\alpha_{k}^2 < 1\}$:
	{\small\begin{align}
		\label{eq:update_rule4}
		&\mathbb{E}\left[\left\|\boldsymbol{\xi}(k)\right\|^{2}|\mathcal{F}_{k}\right]
		\le \left(1+\mu\alpha_{k}\right)(1-\alpha_k\mu)^2\left\|\boldsymbol{\zeta}(k)\right\|^2\nonumber\\
		&+\left(1+\frac{1}{\mu\alpha_{k}}\right)\alpha_{k}^{2}\|\nabla F(\mathbf{x}^{o})\|^2\nonumber\\
		&\le (1-\alpha_k\mu)\left\|\boldsymbol{\zeta}(k)\right\|^2+2\frac{\alpha_k}{\mu}\|\nabla F(\mathbf{x}^{o})\|^2.
		\end{align}}
	
	Using \eqref{eq:update_rule4} in \eqref{eq:eqn-combine-2}, we have for all $k\geq k_0$
	{\small\begin{align}
		\label{eq:combine-3}
		&\mathbb{E}[\,\|\boldsymbol{\zeta}(k+1)\|^2 \,|\,\mathcal{F}_k\,]\nonumber\\&
		\le \left(1-\alpha_k\mu+\frac{\sqrt{N}}{4}s_{1}(P)M\alpha_{k}c_{k}^2+2\frac{d\alpha_{k}^{2}}{c_{k}^{2}}c_{v}\right.\nonumber\\&\left.+4\alpha_{k}^{2}c_{k}^{2}Ns_{1}(P)L^{2}\right)\times\left\|\boldsymbol{\zeta}(k)\right\|^2\nonumber\\
		&+\frac{d\alpha_{k}^{2}}{c_{k}^{2}}\left(2c_{v}\left\|\mathbf{x}^{o}\right\|^{2}+N\sigma_{v}^{2}\right)+\frac{\sqrt{N}}{4}s_{1}(P)L\alpha_{k}c_{k}^2\nonumber\\
		&+\frac{N}{16}s_{1}^{2}(P)M^{2}\alpha_{k}^{2}c_{k}^{4}+2\frac{\alpha_k}{\mu}\|\nabla F(\mathbf{x}^{o})\|^2+2N\alpha_{k}^{2}c_{k}^{4}s_{2}(P)\nonumber\\&+\alpha_{k}^{2}c_{k}^2\frac{\sqrt{N}}{2}s_{1}(P)M\left\|\nabla F(\mathbf{x}^{o})\right\|+4\alpha_{k}^{2}c_{k}^{2}Ns_{1}(P)\left\|\nabla F\left(\mathbf{x}^{o}\right)\right\|^{2}.
		\end{align}}
	
	\noindent Define $k_1$ as follows:
	{\small\begin{align*}
	k_{1}=\inf\left\{k | \frac{\mu}{2} > \frac{\sqrt{N}}{4}s_{1}(P)Mc_{k}^2+\frac{2dc_{v}\alpha_{k}}{c_{k}^{2}}+4\alpha_{k}c_{k}^{2}Ns_{1}(P)L^{2}\right\}.
	\end{align*}}
	It is to be noted that $k_1$ is necessarily finite as $c_{k}\to 0$ and $\alpha_{k}c_{k}^{-2}\to 0$ as $k\to\infty$. We proceed by using the following auxiliary lemma.
	\begin{Lemma}\label{lemma:technical0}
			Let $a_k\in (0,1)$, $u\leq 0$ and $d_k \geq 0$, for all $k\geq 1$. If $q_{k_0}\geq 0$ and for all $k\geq k_0$ there holds $q_{k+1}\leq (1-a_k)q_k + a_k u + d_k$, then, for all $k\geq k_0$,
			\begin{equation}\label{eq:lemmatechnical0_result}
			q_{k+1} \leq q_{k_0}+ u + \sum_{l=l_0}^k d_l.
			\end{equation}
		\end{Lemma}
		\begin{IEEEproof}
			Introduce $p(k,l)=(1-a_k)\cdots (1-a_l)$, for $l\leq k$ and also $p(k,k+1)=1$. It is easy to see that, for every $k\geq k_0$, $q_{k+1}\leq p(k,k_0) q_{k_0}+ u \sum_{l=k_0}^k p(k,l+1)a_l + \sum_{l=k_0}^k p(k,l+1) d_l$. Note now that $p(k,l+1)a_l = p(k,l+1)-p(k,l)$, and hence $\sum_{l=k_0}^k p(k,l+1)a_l = 1- p(k,k_0)\leq 1$. Using the latter together with the fact that $p(k,l+1)\leq 1$ proves the claim of the lemma.
		\end{IEEEproof} 
		Applying Lemma~\ref{lemma:technical0} to $q_k=\mathbb{E}\left[\|\boldsymbol{\zeta}(k)\|^2 \right]$, $a_k=\frac{\mu\alpha_k}{2}$, $u=4 \frac{\|\nabla F(x^o)\|^2}{\mu^2}$, and $d_k$ defined as the remaining term in~\eqref{eq:combine-3} we have, $\forall k \geq \max\left\{k_0,k_1\right\}\doteq k_2$,
		{\small\begin{align}
			\label{eq:combine-4} &\mathbb{E}\left[\|\boldsymbol{\zeta}(k+1)\|^2 \right] \le q_{k_2}(N,d,\alpha_0,c_0)+4\frac{\|\nabla F(\mathbf{x}^{o})\|^2}{\mu^2}\nonumber\\&+\frac{\sqrt{N}s_{1}(P)M\alpha_0c_0^2}{8\delta}+\frac{Ns_{1}^{2}(P)M^{2}\alpha_0^2c_0^4}{16(1+4\delta)}\nonumber\\&+\frac{d\alpha_0^2\left(2c_{v}N\left\|\mathbf{x}^{o}\right\|^{2}+N\sigma_{v}^{2}\right)}{c_0^{2}(1-2\delta)}+\frac{\alpha_0^2c_0^2\sqrt{N}s_{1}(P)L\left\|\nabla F(\mathbf{x}^{o})\right\|}{1+2\delta}\nonumber\\&+\frac{2N\alpha_{0}^{2}c_{0}^{4}s_{2}(P)}{1+4\delta}+\frac{4\alpha_{0}^{2}c_{0}^{2}Ns_{1}(P)}{1+2\delta}\left\|\nabla F\left(\mathbf{x}^{o}\right)\right\|^{2}\nonumber\\&\doteq q_{\infty}(N,d,\alpha_0,c_0),
			\end{align}}
	%	Then, we have, $\forall k \geq \max\left\{k_0,k_1\right\}$,
	%	{\small\begin{align}
	%		\label{eq:combine-4}
	%		&\mathbb{E}\left[\|\boldsymbol{\zeta}(k+1)\|^2 \right] \le \prod_{l=k_0}^{k}\left(1-\frac{\mu\alpha_{l}}{2}\right)\mathbb{E}\left[\|\boldsymbol{\zeta}(k_0)\|^2\right]\nonumber\\
	%		&+4\frac{\|\nabla F(\mathbf{x}^{o})\|^2}{\mu^2}\sum_{l=k_0}^{k}\left(\prod_{m=l+1}^{k}\left(1-\frac{\mu\alpha_{m}}{2}\right)-\prod_{m=l}^{k}\left(1-\frac{\mu\alpha_{m}}{2}\right)\right)
	%		\nonumber\\&+\sum_{l=k_0}^{k}\left(\frac{d\alpha_{l}^{2}}{c_{l}^{2}}\left(2c_{v}\left\|\mathbf{x}^{o}\right\|^{2}+N\sigma_{v}^{2}\right)+\frac{\sqrt{N}}{4}s_{1}(P)M\alpha_{l}c_{l}^{2}
	%		+\frac{N}{16}s_{1}^{2}(P)M^{2}\alpha_{l}^{2}c_{l}^{4}+\alpha_{l}^{2}c_{l}^{2}\frac{\sqrt{N}}{2}s_{1}(P)M\left\|\nabla F(\mathbf{x}^{o})\right\|\right)\nonumber\\
	%		&\Rightarrow \mathbb{E}\left[\|\boldsymbol{\zeta}(k+1)\|^2 \right] \le q_{k_0}(N,d,\alpha_0,c_0)+4\frac{\|\nabla F(\mathbf{x}^{o})\|^2}{\mu^2}\nonumber\\&+\frac{\sqrt{N}s_{1}(P)M\alpha_0c_0^2}{8\delta}+\frac{Ns_{1}^{2}(P)M^{2}\alpha_0^2c_0^4}{16(1+4\delta)}\nonumber\\&+\frac{d\alpha_0^2\left(2c_{v}N\left\|\mathbf{x}^{o}\right\|^{2}+N\sigma_{v}^{2}\right)}{c_0^{2}(1-2\delta)}+\frac{\alpha_0^2c_0^2\sqrt{N}s_{1}(P)M\left\|\nabla F(\mathbf{x}^{o})\right\|}{1+2\delta}\nonumber\\&\doteq q_{\infty}(N,d,\alpha_0,c_0),
	%		\end{align}}
\end{proof}
\noindent From \eqref{eq:combine-4}, we have that $\mathbb{E}\left[\left\|\mathbf{x}(k+1)-\mathbf{x}^{o}\right\|^{2}\right]$ is finite and bounded from above, where $\mathbb{E}\left[\left\|\mathbf{x}(k_2)-\mathbf{x}^{o}\right\|^{2}\right] \le q_{k_2}(N,d,\alpha_0,c_0)$. From the boundedness of $\mathbb{E}\left[\left\|\mathbf{x}(k)-\mathbf{x}^{o}\right\|^2\right]$, we have also established the boundedness of $\mathbb{E}\left[\left\|\nabla F(\mathbf{x}(k))\right\|^2\right]$ and $\mathbb{E}\left[\left\|\mathbf{x}(k)\right\|^2\right]$. 

\noindent With the above development in place, we can bound the variance of the noise process $\{\mathbf{v}_{\mathbf{z}}(k;\mathbf{x}(k))\}$ as follows:
{\small\begin{align}
\label{eq:noise_variance_condition}
&\mathbb{E}\left[\left\|\mathbf{v}_{\mathbf{z}}(k;\mathbf{x}(k))\right\|^{2}|\mathcal{F}_{k}\right]\le 2dc_{v}q_{\infty}(N,d,\alpha_0,c_0)\nonumber\\&+2Nd\underbrace{\left(\sigma_{v}^{2}+\left\|\mathbf{x}^{\ast}\right\|^{2}\right)}_{\text{$\sigma_{1}^{2}$}}.
\end{align}}

\noindent We also have the following bound:
{\small\begin{align*}
	&\mathbb{E}\left[\left\|\mathbf{g}(\mathbf{x}(k))-\mathbb{E}\left[\widehat{\mathbf{g}}(\mathbf{x}(k))\mid\mathcal{F}_{k}\right]\right\|^{2}\mid\mathcal{F}_{k}\right]\nonumber\\
	&\le 4Ns_{1}(P)L^{2}q_{\infty}(N,d,\alpha_0,c_0)+4Ns_{1}(P)\left\|\nabla F\left(\mathbf{x}^{o}\right)\right\|^{2}+2Nc_{k}^{2}s_{2}(P).
	\end{align*}}

\noindent We now study the disagreement of the optimizer sequence $\{\mathbf{x}_{i}(k)\}$ at a node $i$ with respect to the~(hypothetically available) network averaged optimizer sequence, i.e., $\overline{\mathbf{x}}(k)=\frac{1}{N}\sum_{i=1}^{N}\mathbf{x}_{i}(k)$. Define the disagreement at the $i$-th node as $\widetilde{\mathbf{x}}_{i}(k)=\mathbf{x}_{i}(k)-\overline{\mathbf{x}}(k)$. The vectorized version of the disagreements $\widetilde{\mathbf{x}}_{i}(k),~i=1,\cdots,N$, can then be written as $\widetilde{\mathbf{x}}(k)=\left(\mathbf{I}-\mathbf{J}\right)\mathbf{x}(k)$, where $\mathbf{J}=\frac{1}{N}\left(\mathbf{1}_{N}\otimes\mathbf{I}_{d}\right)\left(\mathbf{1}_{N}\otimes\mathbf{I}_{d}\right)^{\top}=\frac{1}{N}\mathbf{1}_{N}\mathbf{1}_{N}^{\top}\otimes\mathbf{I}_{d}$.
We have the following Lemma:
\begin{Lemma}
	\label{lemma-0disag-bound}
	Let the hypotheses of Theorem \ref{theorem-1} hold. Then, we have
	{\small\begin{align*}
	&\mathbb{E}\left[\left\|\widetilde{\mathbf{x}}(k+1)\right\|^{2}\right] \le Q_{k}+ \frac{4\Delta_{1,\infty}\alpha_{0}^{2}}{\lambda_{2}^{2}\left(\overline{\mathbf{R}}\right)\beta_0^{2}c_0^2(k+1)^{2-2\tau-2\delta}}\nonumber\\&=O\left(\frac{1}{k^{2-2\delta-2\tau}}\right),
	\end{align*}}
	where $Q_k$ is a term which decays faster than $(k+1)^{-2+2\tau+2\delta}$.
\end{Lemma}
\noindent Lemma~\ref{lemma-0disag-bound} plays a crucial role in providing a tight bound for the bias in the gradient estimates
according to which the global average $\overline{\mathbf{x}}(k)$ evolves.
\begin{proof}
	The process $\{\widetilde{\mathbf{x}}(k)\}$ follows the recursion:
	{\small\begin{align}
	\label{eq:0dis1}
	&\widetilde{\mathbf{x}}(k+1)=\widetilde{\mathbf{W}}_{k}\widetilde{\mathbf{x}}(k)\nonumber\\&-\frac{\alpha_{k}}{c_{k}}\left(\mathbf{I}-\mathbf{J}\right)\underbrace{\left(c_{k}\nabla F(\mathbf{x}(k))+c_{k}\mathbf{h}(\mathbf{x}(k))+c_{k}^{2}\mathbf{b}\left(\mathbf{x}(k)\right)\right)}_{\text{$\mathbf{w}(k)$}},
	\end{align}}
	where $\widetilde{\mathbf{W}}_{k} = \mathbf{W}_{k}-\mathbf{J}$. 
	%We also have that from Lemma {\color{red} Fill up Lemma number}, $\mathbb{E}\left[\left\|\mathbf{w}(k)\right\|^{2}|\right]\leq c_{7} < \infty$, which follows due to the boundedness of $\mathbb{E}\left[\left\|\nabla F(\mathbf{x}(k))\right\|^2\right]$. 
	Then, we have,
	{\small\begin{align}
	\label{eq:0dis2}
	\left\|\widetilde{\mathbf{x}}(k+1)\right\| \le \left\|\widetilde{\mathbf{W}}_{k}\widetilde{\mathbf{x}}(k)\right\|+\frac{\alpha_{k}}{c_k}\left\|\mathbf{w}(k)\right\|.
	\end{align}}
	Using \eqref{eq:cool_ineq_1} in \eqref{eq:0dis1}, we have,
	{\small\begin{align}
	\label{eq:dis5}
	&\left\|\widetilde{\mathbf{x}}(k+1)\right\|^{2}\le \left(1+\theta_k\right)\left\|\widetilde{\mathbf{W}}_{k}\widetilde{\mathbf{x}}(k)\right\|^{2}\nonumber\\&+\left(1+\frac{1}{\theta_k}\right)\frac{\alpha_k^2}{c_{k}^{2}}\left\|\widetilde{\mathbf{w}}(k)\right\|^2.
	\end{align}}
	We, now bound the term $\mathbb{E}\left[\left\|\widetilde{\mathbf{W}}_{k}\widetilde{\mathbf{x}}(k)\right\|^{2}|\mathcal{F}_{k}\right]$.
	{\small\begin{align}
	\label{eq:0dis5_bound}
	&\mathbb{E}\left[\left\|\widetilde{\mathbf{W}}(k)\widetilde{\mathbf{x}}(k)\right\|^{2}|\mathcal{F}_{k}\right] = \widetilde{\mathbf{x}}^{\top}(k)\mathbb{E}\left[\widetilde{\mathbf{W}}^{2}(k)-\mathbf{J}|\mathcal{F}_{k}\right]\widetilde{\mathbf{x}}(k) \nonumber\\
	&=\widetilde{\mathbf{x}}^{\top}(k)\left(\mathbf{I}-2\beta_{k}\overline{\mathbf{R}}+\beta_{k}^{2}\overline{\mathbf{R}}^{2}+\widetilde{\mathbf{R}}(k)^{2}-\mathbf{J}\right)\widetilde{\mathbf{x}}(k) \nonumber\\
	&\le \left(1-2\beta_{k}\lambda_{2}\left(\overline{\mathbf{R}}\right)+\beta_{k}^{2}\lambda_{N}^2\left(\overline{\mathbf{R}}\right)\right.\nonumber\\&\left.+\frac{4N^{2}\beta_{0}\rho_{0}^{2}}{(k+1)^{\tau+\epsilon}}-4\beta_{k}^{2}N^{2}\right)\left\|\widetilde{\mathbf{x}}(k)\right\|^{2} \nonumber\\
	&\le \left(1-2\beta_{k}\lambda_{2}\left(\overline{\mathbf{R}}\right)+\frac{4N^{2}\beta_{0}\rho_{0}^{2}}{(k+1)^{\tau+\epsilon}}\right)\left\|\widetilde{\mathbf{x}}(k)\right\|^{2} \nonumber\\
	&\le \left(1-\beta_{k}\lambda_{2}\left(\overline{\mathbf{R}}\right)\right)\left\|\widetilde{\mathbf{x}}(k)\right\|^{2},
	\end{align}}
	where the last inequality follows from assumption \ref{as:a3}.
	Then, we have,
	{\small\begin{align}
	\label{eq:dis6}
	&\mathbb{E}\left[\left\|\widetilde{\mathbf{x}}(k+1)\right\|^{2}|\mathcal{F}_{k}\right]\leq \left(1+\theta_k\right)(1-\beta_{k}\lambda_{2}\left(\overline{\mathbf{R}}\right))\left\|\widetilde{\mathbf{x}}(k)\right\|^{2}\nonumber\\&+\left(1+\frac{1}{\theta_k}\right)\frac{\alpha_k^2}{c_k^{2}}\mathbb{E}\left[\left\|\mathbf{w}(k)\right\|^{2}|\mathcal{F}_{k}\right],
	\end{align}}
	where 
	{\small\begin{align}
	\label{eq:wk_bound}
	&\mathbb{E}\left[\left\|\mathbf{w}(k)\right\|^{2}|\mathcal{F}_{k}\right] \le 3c_{k}^{2}\left\|\nabla F(\mathbf{x}(k))\right\|^{2}+3c_{k}^{2}\mathbb{E}\left[\left\|\mathbf{h}(\mathbf{x}(k))\right\|^{2}|\mathcal{F}_{k}\right]\nonumber\\&+3c_{k}^{2}\left\|\mathbf{b}\left(\mathbf{x}(k)\right)\right\|^{2}\nonumber\\
	&\le 3c_{k}^{2}\left\|\nabla F(\mathbf{x}(k))\right\|^{2}+\frac{3}{16}c_{k}^{4}Ns_{1}^{2}(P)M^{2}\nonumber\\&+6dc_{v}q_{\infty}(N,d,\alpha_0,c_0)+6dN\sigma_{1}^{2}+6Nc_{k}^{4}s_{2}(P)\nonumber\\
	&+12c_{k}^{2}Ns_{1}(P)L^{2}q_{\infty}(N,d,\alpha_0,c_0)+12c_{k}^{2}Ns_{1}(P)\left\|\nabla F\left(\mathbf{x}^{o}\right)\right\|^{2}\nonumber\\
	&\Rightarrow \mathbb{E}\left[\left\|\mathbf{w}(k)\right\|^{2}\right] \le 3\left(2dc_{v}+c_{k}^{2}L^{2}(1+4Ns_{1}(P))\right)\nonumber\\&\times q_{\infty}(N,d,\alpha_0,c_0)\nonumber\\&+\frac{3}{16}c_{k}^{4}Ns_{1}^{2}(P)M^{2}+6Nc_{k}^{4}s_{2}(P)\nonumber\\&+6dN\sigma_{1}^{2}+12c_{k}^{2}Ns_{1}(P)\left\|\nabla F\left(\mathbf{x}^{o}\right)\right\|^{2}\nonumber\\
	%&= \underbrace{6dc_{v}q_{\infty}(N,d,\alpha_0,c_0)+6dN\sigma_1^{2}}_{\text{$\Delta_{1,\infty}$}}\nonumber\\&+\underbrace{\frac{3}{16}c_{k}^{4}Ns_{1}^{2}(P)M^{2}+3c_{k}^{2}L^{2}q_{\infty}(N,d,\alpha_0,c_0)}_{\text{$c_{k}^{2}\Delta_{2,\infty}$}}
	&=\Delta_{1,\infty}+c_{k}^{2}\Delta_{2,\infty}\doteq\Delta_{k}\nonumber\\&\Rightarrow \mathbb{E}\left[\left\|\mathbf{w}(k)\right\|^{2}\right] < \infty,
	\end{align}}
	where {\small$\Delta_{1,\infty}=6dc_{v}q_{\infty}(N,d,\alpha_0,c_0)+6dN\sigma_1^{2}$} and {\small$c_{k}^{2}\Delta_{2,\infty}=\frac{3}{16}c_{k}^{4}Ns_{1}^{2}(P)M^{2}+3c_{k}^{2}L^{2}(1+4Ns_{1}(P))q_{\infty}(N,d,\alpha_0,c_0)+12c_{k}^{2}Ns_{1}(P)\left\|\nabla F\left(\mathbf{x}^{o}\right)\right\|^{2}+6Nc_{k}^{4}s_{2}(P)$}.
	With the above development in place, we then have,
	{\small\begin{align}
	\label{eq:dis6.5}
	&\mathbb{E}\left[\left\|\widetilde{\mathbf{x}}(k+1)\right\|^{2}\right]\leq \left(1+\theta_k\right)\left(1-\beta_{k}\lambda_{2}\left(\overline{\mathbf{R}}\right)\right)\left\|\widetilde{\mathbf{x}}(k)\right\|^{2}\nonumber\\&+\left(1+\frac{1}{\theta_k}\right)\frac{\alpha_k^2}{c_k^{2}}\Delta_{k}.
	\end{align}}
	In particular, we choose $\theta(k) = \frac{\beta_{k}}{2}\lambda_{2}\left(\overline{\mathbf{R}}\right)$. From \eqref{eq:dis6.5}, we have,
	{\small\begin{align}
	\label{eq:dis6.6}
	&\mathbb{E}\left[\left\|\widetilde{\mathbf{x}}(k+1)\right\|^{2}\right]\leq \left(1-\frac{\beta_{k}}{2}\lambda_{2}\left(\overline{\mathbf{R}}\right)\right)\mathbb{E}\left[\left\|\widetilde{\mathbf{x}}(k)\right\|^{2}\right]\nonumber\\&+\left(1+\frac{2}{\beta_{k}\lambda_{2}\left(\overline{\mathbf{R}}\right)}\right)\frac{\alpha_k^2}{c_{k}^{2}}\Delta_{k}\nonumber\\
	&= \left(1-\frac{\beta_{k}}{2}\lambda_{2}\left(\overline{\mathbf{R}}\right)\right)\mathbb{E}\left[\left\|\widetilde{\mathbf{x}}(k)\right\|^{2}\right]+\frac{2\alpha_k^2}{\lambda_{2}\left(\overline{\mathbf{R}}\right)c_{k}^{2}\beta_{k}}\Delta_{k}+\frac{\alpha_k^2}{c_{k}^{2}}\Delta_{k}.
	\end{align}}
	For ease of analysis, define $s(k)=\frac{\beta_{k}}{2}\lambda_{2}\left(\overline{\mathbf{R}}\right)$.
	\noindent We proceed by using the following technical lemma.
			\begin{Lemma}\label{lemma:technical1}
				If for all $k\geq k_0$ there holds
				\begin{equation}\label{eq:techlemma2_condition}
				q_{k+1}\leq (1-s_k) q_k +  \left(1+\frac{1}{s_k}\right)b_k d_k,
				\end{equation}
				where $q_{k_0}\geq 0$, $s_k \in (0,1),\,d_k,\,b_k\geq 0$ are monotonously decreasing, then, for any $k\geq m(k)\geq k_0$
				{\small\begin{align}\label{eq:techlemma2_result}
				&q_{k+1} \leq  e^{-\sum_{l=k_0}^k s_l}q_{k_0} + d_{k_0} e^{-\sum_{l=m(k)}^k s_l}  \sum_{l=k_0}^{m(k)-1} \left(1+\frac{1}{s_l}\right) b_l
				\nonumber\\&+ d_{m(k)} b_{m(k)}  \frac{s_k+1}{s_k^2}.
				\end{align}}
			\end{Lemma}
			
			\begin{IEEEproof}
				Similarly as before, define $p(k,l)=(1-s_k)\cdots(1-s_l)$ for $k_0\leq l\leq k$, and let also $p(k,k+1)=1$. Recall that $p(k,l+1)s_l$ can be expressed as $p(k,l+1) s_l=p (k,l+1)-p(k,l)$. Then, we have:
				{\small\begin{align}\label{eq-tech-lemma2_proof}
				q_{k+1} &\leq p(k,k_0) q_{k_0} + \sum_{l=k_0}^k p(k,l)\left(1+\frac{1}{s_l} b_l d_l\right)\\
				& \leq p(k,k_0) q_{k_0} + d_{k_0} p(k,m(k)) \sum_{l=k_0}^{m(k)} \left(1+\frac{1}{s_l}\right) b_l \nonumber\\&+ b_{m(k)} d_{m(k)}\frac{s_k+1}{s_k^2} \sum_{m(k)}^k\left(p(k,l+1)-p(k,l)\right)\nonumber,
				\end{align}}
				where we break the sum in~\eqref{eq-tech-lemma2_proof} at $l=m(k)$, and use the fact that $p\left(k, m(k)-1\right)\geq p(k,l)$ for every $l\leq m(k)-1$, together with the fact that $1/s_l \leq 1/s_k$, for every $l\leq k$. Finally, noting that, for every $l\leq k$, $p(k,l) \leq e^{-\sum_{m=1}^k s_l}$, and also recalling that $\sum_{m(k)}^k\left(p(k,l+1)-p(k,l)\right) \leq 1$, proves the claim of the lemma.
			\end{IEEEproof}
			Applying the preceding lemma to $q_k=\mathbb{E}\left[\left\|\widetilde{\mathbf{x}}(k)\right\|^{2}\right]$, $d_k=\Delta_k$, $b_k=\frac{\alpha_k^2}{c_k^2}$, and $s_k=\frac{\beta_{k}}{2}\lambda_{2}\left(\overline{\mathbf{R}}\right)$ we have,
	{\small\begin{align}
	\label{eq:dis6.7}
	&\mathbb{E}\left[\left\|\widetilde{\mathbf{x}}(k+1)\right\|^{2}\right]\nonumber\\
	&\le \underbrace{\exp\left(-\sum_{l=0}^{k}s(l)\right)\mathbb{E}\left[\left\|\widetilde{\mathbf{x}}(0)\right\|^{2}\right]}_{\text{$t_1$}}\nonumber\\
	&+\underbrace{\Delta_{0}\exp\left(-\sum_{m=\lfloor\frac{k-1}{2}\rfloor}^{k}s(m)\right)\sum_{l=0}^{\lfloor\frac{k-1}{2}\rfloor-1}\left(\frac{2\alpha_l^2}{\lambda_{2}\left(\overline{\mathbf{R}}\right)c_{l}^{2}\beta_{l}}+\frac{\alpha_l^2}{c_{l}^{2}}\right)}_{\text{$t_2$}}\nonumber\\
	&+\underbrace{\frac{4\Delta_{\lfloor\frac{k-1}{2}\rfloor}\alpha_{0}^{2}}{\lambda_{2}^{2}\left(\overline{\mathbf{R}}\right)\beta_0^{2}c_0^2(k+1)^{2-2\tau-2\delta}}}_{\text{$t_3$}}+\underbrace{\frac{2\Delta_{\lfloor\frac{k-1}{2}\rfloor}\alpha_{0}^{2}}{\lambda_{2}\left(\overline{\mathbf{R}}\right)\beta_0c_0^2(k+1)^{2-\tau-2\delta}}}_{\text{$t_4$}}.
	\end{align}}
	In the above proof, the splitting in the interval $[0,k]$ was done at $\lfloor\frac{k-1}{2}\rfloor$ for ease of book keeping. The division can be done at an arbitrary point.
	It is to be noted that the sequence $\{s(k)\}$ is not summable and hence terms $t_1$ and $t_2$ decay faster than $(k+1)^{2-2\tau-2\delta}$. Also, note that term $t_4$ decays faster than $t_3$. Furthermore, $t_3$ can be written as
	{\small\begin{align*}
		&\frac{4\Delta_{\lfloor\frac{k-1}{2}\rfloor}\alpha_{0}^{2}}{\lambda_{2}^{2}\left(\overline{\mathbf{R}}\right)\beta_0^{2}c_0^2(k+1)^{2-2\tau-2\delta}} = \underbrace{\frac{4\Delta_{1,\infty}\alpha_{0}^{2}}{\lambda_{2}^{2}\left(\overline{\mathbf{R}}\right)\beta_0^{2}c_0^2(k+1)^{2-2\tau-2\delta}}}_{\text{$t_{31}$}}\nonumber\\&+\underbrace{\frac{4c_{\lfloor\frac{k-1}{2}\rfloor}^{2}\Delta_{2,\infty}\alpha_{0}^{2}}{\lambda_{2}^{2}\left(\overline{\mathbf{R}}\right)\beta_0^{2}c_0^2(k+1)^{2-2\tau-2\delta}}}_{\text{$t_{32}$}},
		\end{align*}}
	from which we have that $t_{32}$ decays faster than $t_{31}$.
	For notational ease, henceforth we refer to $t_1+t_2+t_{32}+t_4=Q_{k}$, while keeping in mind that $Q_{k}$ decays faster than $(k+1)^{2-2\tau-2\delta}$.
	%\end{proof}
	Hence, we have the disagreement given by,
	{\small\begin{align*}
	\mathbb{E}\left[\left\|\widetilde{\mathbf{x}}(k+1)\right\|^{2}\right]  =O\left(\frac{1}{k^{2-2\delta-2\tau}}\right).
	\end{align*}}
\end{proof}
We now proceed to the proof of Theorem \ref{theorem-1}.
\noindent Denote $\overline{\mathbf{x}}(k)=\frac{1}{N}\sum_{n=1}\mathbf{x}_{i}(k)$.
Then, we have,
{\small\begin{align}
	\label{eq:opt1}
	&\overline{\mathbf{x}}(k+1) = \overline{\mathbf{x}}(k)\nonumber\\&-\frac{\alpha_{k}}{c_k}\left[\frac{c_k}{N}\sum_{i=1}^{N}\nabla f_{i}\left(\mathbf{x}_{i}(k)\right)+\underbrace{\frac{c_k^{2}}{N}\sum_{i=1}^{N}\mathbf{b}_{i}\left(\mathbf{x}_{i}(k)\right)}_{\text{$\overline{\mathbf{b}}\left(\mathbf{x}(k)\right)$}}\right.\nonumber\\&\left.+\underbrace{\frac{c_{k}}{N}\sum_{i=1}^{N}\mathbf{h}_{i}(\mathbf{x}_{i}(k))}_{\text{$\overline{\mathbf{h}}(\mathbf{x}(k))$}}\right]\nonumber\\
	&\Rightarrow \overline{\mathbf{x}}(k+1) = \overline{\mathbf{x}}(k)-\frac{\alpha_{k}}{c_{k}}\left(\overline{\mathbf{h}}(\mathbf{x}(k))+\overline{\mathbf{b}}\left(\mathbf{x}(k)\right)\right)\nonumber\\&-\frac{\alpha_k}{Nc_{k}}\left[c_k\sum_{i=1}^{N}\nabla f_{i}\left(\mathbf{x}_{i}(k)\right)-\nabla f_{i}\left(\overline{\mathbf{x}}(k)\right)+\nabla f_{i}\left(\overline{\mathbf{x}}(k)\right)\right].
	\end{align}}
Recall that $f(\cdot)=\sum_{i=1}^{N}f_{i}(\cdot)$.
Then, we have,
{\small\begin{align}
\label{eq:0opt2}
&\overline{\mathbf{x}}(k+1) = \overline{\mathbf{x}}(k)-\frac{\alpha_{k}}{c_{k}}\left(\overline{\mathbf{h}}(\mathbf{x}(k))+\overline{\mathbf{b}}\left(\mathbf{x}(k)\right)\right)\nonumber\\&-\frac{\alpha_{k}}{N}\nabla f\left(\overline{\mathbf{x}}(k)\right)-\frac{\alpha_k}{N}\left[\sum_{i=1}^{N}\nabla f_{i}\left(\mathbf{x}_{i}(k)\right)-\nabla f_{i}\left(\overline{\mathbf{x}}(k)\right)\right]\nonumber\\
&\Rightarrow \overline{\mathbf{x}}(k+1) = \overline{\mathbf{x}}(k)-\frac{\alpha_k}{Nc_{k}}\left[c_{k}\nabla f\left(\overline{\mathbf{x}}(k)\right)+\mathbf{e}(k)\right],
\end{align}}
where
{\small\begin{align}
\label{eq:0opt3}
&\mathbf{e}(k) = N\overline{\mathbf{h}}(\mathbf{x}(k))\nonumber\\&+\underbrace{N\overline{\mathbf{b}}\left(\mathbf{x}(k)\right)+c_{k}\sum_{i=1}^{N}\left(\nabla f_{i}\left(\mathbf{x}_{i}(k)\right)-\nabla f_{i}\left(\overline{\mathbf{x}}(k)\right)\right)}_{\text{$\boldsymbol{\epsilon}(k)$}}.
\end{align}}
Note that, $c_{k}\left\|\nabla f_{i}\left(\mathbf{x}_{i}(k)\right)-\nabla f_{i}\left(\overline{\mathbf{x}}(k)\right)\right\| \leq c_{k}L\left\|\mathbf{x}_{i}(k)-\overline{\mathbf{x}}(k)\right\| = c_{k}L\left\|\widetilde{\mathbf{x}}_{i}(k)\right\|$. We also have that, $\left\|\overline{\mathbf{b}}\left(\mathbf{x}(k)\right)\right\| \le \frac{M}{4}s_{1}(P)c_{k}^{3}$. Thus, we can conclude that, $\forall k\geq k_3$
{\small\begin{align}
\label{eq:opt4}
&\boldsymbol{\epsilon}(k) = c_{k}\sum_{i=1}^{N}\left(\nabla f_{i}\left(\mathbf{x}_{i}(k)\right)-\nabla f_{i}\left(\overline{\mathbf{x}}(k)\right)\right)+N\overline{\mathbf{b}}\left(\mathbf{x}(k)\right)\nonumber\\
&\Rightarrow\left\|\boldsymbol{\epsilon}(k)\right\|^{2} \leq 2NL^{2}c_{k}^{2}\left\|\widetilde{\mathbf{x}}(k)\right\|^{2}+\frac{N}{8}M^{2}d^2(P)c_{k}^{6}\nonumber\\
&\Rightarrow\mathbb{E}\left[\left\|\boldsymbol{\epsilon}(k)\right\|^{2}\right] \leq \frac{8NL^{2}\Delta_{1,\infty}\alpha_{0}^{2}}{\lambda_{2}^{2}\left(\overline{\mathbf{R}}\right)\beta_0^{2}(k+1)^{2-2\tau}}+\frac{NM^{2}d^2(P)c_{0}^{6}}{8(k+1)^{6\delta}}\nonumber\\&+\frac{2NL^{2}Q_{k}c_{0}^{2}}{(k+1)^{2\delta}}.
\end{align}}
With the above development in place, we rewrite \eqref{eq:0opt2} as follows:
{\small\begin{align}
\label{eq:0opt5}
&\overline{\mathbf{x}}(k+1) = \overline{\mathbf{x}}(k)-\frac{\alpha_{k}}{N}\nabla f\left(\overline{\mathbf{x}}(k)\right)-\frac{\alpha_k}{Nc_{k}}\boldsymbol{\epsilon}(k)-\frac{\alpha_{k}}{c_{k}}\overline{\mathbf{h}}(\mathbf{x}(k))\nonumber\\
&\Rightarrow \overline{\mathbf{x}}(k+1)-\mathbf{x}^{\ast} = \overline{\mathbf{x}}(k)-\mathbf{x}^{\ast}-\frac{\alpha_{k}}{N}\left[\nabla f\left(\overline{\mathbf{x}}(k)\right)-\underbrace{\nabla f\left(\mathbf{x}^{\ast}\right)}_{\text{$=0$}}\right]\nonumber\\&-\frac{\alpha_k}{Nc_{k}}\boldsymbol{\epsilon}(k)-\frac{\alpha_{k}}{c_{k}}\overline{\mathbf{h}}(\mathbf{x}(k)).
\end{align}}
By Leibnitz rule, we have,
{\small\begin{align}
\label{eq:0opt6}
&\nabla f\left(\overline{\mathbf{x}}(k)\right)-\nabla f\left(\mathbf{x}^{\ast}\right) \nonumber\\&= \underbrace{\left[\int_{s=0}^{1}\nabla^{2}f\left(\mathbf{x}^{\ast}+s\left(\overline{\mathbf{x}}(k)-\mathbf{x}^{\ast}\right)\right)ds\right]}_{\text{$\overline{\mathbf{H}}_{k}$}}\left(\overline{\mathbf{x}}(k)-\mathbf{x}^{\ast}\right),
\end{align}}
where it is to be noted that $NL\succcurlyeq\overline{\mathbf{H}}_{k}\succcurlyeq N\mu$.
Using \eqref{eq:0opt6} in  \eqref{eq:0opt5}, we have,
{\small\begin{align}
\label{eq:0opt7}
&\left(\overline{\mathbf{x}}(k+1)-\mathbf{x}^{\ast}\right) = \left[\mathbf{I}-\frac{\alpha_k}{N}\overline{\mathbf{H}}_{k}\right]\left(\overline{\mathbf{x}}(k)-\mathbf{x}^{\ast}\right)\nonumber\\&-\frac{\alpha_k}{Nc_{k}}\boldsymbol{\epsilon}(k)-\frac{\alpha_{k}}{c_{k}}\overline{\mathbf{h}}(\mathbf{x}(k)).
\end{align}}
Denote by $\mathbf{m}(k)=\left[\mathbf{I}-\frac{\alpha_k}{N}\overline{\mathbf{H}}_{k}\right]\left(\overline{\mathbf{x}}(k)-\mathbf{x}^{\ast}\right)-\frac{\alpha_k}{Nc_{k}}\boldsymbol{\epsilon}(k)$ and note that $\mathbf{m}(k)$ is conditionally independent from $\overline{\mathbf{h}}(\mathbf{x}(k))$ given the history $\mathcal{F}_{k}$. Then \eqref{eq:0opt7} can be rewritten as:
{\small\begin{align}
\label{eq:0opt8}
&\left(\overline{\mathbf{x}}(k+1)-\mathbf{x}^{\ast}\right)=\mathbf{m}(k)-\frac{\alpha_{k}}{c_{k}}\overline{\mathbf{h}}(\mathbf{x}(k))\nonumber\\
&\Rightarrow \left\|\overline{\mathbf{x}}(k+1)-\mathbf{x}^{\ast}\right\|^{2} \le \left\|\mathbf{m}(k)\right\|^{2}-2\frac{\alpha_{k}}{c_{k}}\mathbf{m}(k)^{\top}\overline{\mathbf{h}}(\mathbf{x}(k))\nonumber\\&+\frac{\alpha_{k}^{2}}{c_{k}^{2}}\left\|\overline{\mathbf{h}}(\mathbf{x}(k))\right\|^{2}.
\end{align}}
Using the properties of conditional expectation and noting that $\mathbb{E}\left[\mathbf{h}(\mathbf{x}(k))|\mathcal{F}_{k}\right]=\mathbf{0}$, we have,
{\small\begin{align}
	\label{eq:0opt9}
	&\mathbb{E}\left[\left\|\overline{\mathbf{x}}(k+1)-\mathbf{x}^{\ast}\right\|^{2}|\mathcal{F}_{k}\right] \le \left\|\mathbf{m}(k)\right\|^{2}+\frac{\alpha_{k}^{2}}{c_k^2}\mathbb{E}\left[\left\|\overline{\mathbf{h}}(\mathbf{x}(k))\right\|^{2}|\mathcal{F}_{k}\right]\nonumber\\
	&\Rightarrow \mathbb{E}\left[\left\|\overline{\mathbf{x}}(k+1)-\mathbf{x}^{\ast}\right\|^{2}\right] \le \mathbb{E}\left[\left\|\mathbf{m}(k)\right\|^{2}\right]+2N\alpha_{k}^{2}c_{k}^{2}s_{2}(P)
	\nonumber\\&+\frac{2\alpha_{k}^{2}\left(dc_{v}q_{\infty}(N,d,\alpha_0,c_0)+dN\sigma_1^{2}\right)}{c_k^2}\nonumber\\&+4\alpha_{k}^{2}Ns_{1}(P)L^{2}q_{\infty}(N,d,\alpha_0,c_0)+4\alpha_{k}^{2}Ns_{1}(P)\left\|\nabla F\left(\mathbf{x}^{o}\right)\right\|^{2}.
	\end{align}}
\noindent For notational simplicity, we denote $\alpha_{k}^{2}\sigma_{h}^{2}=2N\alpha_{k}^{2}c_{k}^{2}s_{2}(P)+4\alpha_{k}^{2}Ns_{1}(P)L^{2}q_{\infty}(N,d,\alpha_0,c_0)+4\alpha_{k}^{2}Ns_{1}(P)\left\|\nabla F\left(\mathbf{x}^{o}\right)\right\|^{2}$.
Using \eqref{eq:cool_ineq_1}, we have for $\mathbf{m}(k)$,
{\small\begin{align}
\label{eq:opt10}
&\left\|\mathbf{m}(k)\right\|^{2} \le \left(1+\theta_{k}\right)\left\|\mathbf{I}-\frac{\alpha_k}{N}\overline{\mathbf{H}}_{k}\right\|^{2}\left\|\overline{\mathbf{x}}(k)-\mathbf{x}^{\ast}\right\|^{2}\nonumber\\&+\left(1+\frac{1}{\theta_k}\right)\frac{\alpha_k^2}{N^2c_{k}^{2}}\left\|\boldsymbol{\epsilon}(k)\right\|^{2}\nonumber\\
&\le \left(1+\theta_{k}\right)\left(1-\frac{\mu\alpha_{0}}{k+1}\right)^{2}\left\|\overline{\mathbf{x}}(k)-\mathbf{x}^{\ast}\right\|^{2}\nonumber\\&+\left(1+\frac{1}{\theta_k}\right)\frac{\alpha_k^2}{N^2c_{k}^{2}}\left\|\boldsymbol{\epsilon}(k)\right\|^{2}.
\end{align}}
On choosing $\theta_{k}=\frac{\mu\alpha_0}{k+1}$, where $\alpha_0 > \frac{1}{\mu}$, we have, 
{\small\begin{align}
	\label{eq:0opt11}
	&\mathbb{E}\left[\left\|\mathbf{m}(k)\right\|^{2}\right] \leq \left(1-\frac{\mu\alpha_0}{k+1}\right)\mathbb{E}\left[\left\|\overline{\mathbf{x}}(k)-\mathbf{x}^{\ast}\right\|^{2}\right]\nonumber\\&+
	\frac{16L^{2}\Delta_{1,\infty}N\alpha_{0}^{3}}{\mu\lambda_{2}^{2}\left(\overline{\mathbf{R}}\right)c_0^{2}\beta_0^{2}(k+1)^{3-2\tau-2\delta}}+\frac{4M^{2}Nd^2(P)c_{0}^{4}\alpha_0}{\mu(k+1)^{1+4\delta}}+\frac{4L^{2}NQ_{k}}{\mu(k+1)}\nonumber\\
	&\Rightarrow \mathbb{E}\left[\left\|\overline{\mathbf{x}}(k+1)-\mathbf{x}^{\ast}\right\|^{2}\right] \leq \left(1-\frac{\mu\alpha_0}{k+1}\right)\mathbb{E}\left[\left\|\overline{\mathbf{x}}(k)-\mathbf{x}^{\ast}\right\|^{2}\right]\nonumber\\&+\frac{16NL^{2}\Delta_{1,\infty}\alpha_{0}^{3}}{\mu\lambda_{2}^{2}\left(\overline{\mathbf{R}}\right)c_0^{2}\beta_0^{2}(k+1)^{3-2\tau-2\delta}}+\frac{4NM^{2}d^2(P)c_{0}^{4}\alpha_0}{\mu(k+1)^{1+4\delta}}+\frac{4NL^{2}Q_{k}}{\mu(k+1)}\nonumber\\
	&+\frac{2\alpha_{k}^{2}\left(dc_{v}q_{\infty}(N,d,\alpha_0,c_0)+dN\sigma_1^{2}\right)}{c_k^2}+\alpha_{k}^{2}\sigma_{h}^{2}\nonumber\\
	&\Rightarrow\mathbb{E}\left[\left\|\overline{\mathbf{x}}(k+1)-\mathbf{x}^{\ast}\right\|^{2}\right] \leq \left(1-\frac{\mu\alpha_0}{k+1}\right)\mathbb{E}\left[\left\|\overline{\mathbf{x}}(k)-\mathbf{x}^{\ast}\right\|^{2}\right]\nonumber\\&+\frac{16NL^{2}\Delta_{1,\infty}\alpha_{0}^{3}}{\mu\lambda_{2}^{2}\left(\overline{\mathbf{R}}\right)c_0^{2}\beta_0^{2}(k+1)^{3-2\tau-2\delta}}\nonumber\\&+\frac{4M^{2}Nd^2(P)c_{0}^{4}\alpha_0}{\mu(k+1)^{1+4\delta}}+\frac{2\alpha_{0}^{2}\left(dc_{v}q_{\infty}(N,d,\alpha_0,c_0)+dN\sigma_1^{2}\right)}{c_0^2(k+1)^{2-2\delta}}+P_{k},
	\end{align}}
where $P_{k}=\frac{4NL^{2}Q_{k}}{\mu(k+1)}+\frac{\alpha_0^{2}\sigma_{h}^{2}}{(k+1)^{2}}$ decays faster as compared to the other terms.
Proceeding as in \eqref{eq:dis6.7}, we have
{\small\begin{align}
	\label{eq:opt11.5}
	&\mathbb{E}\left[\left\|\overline{\mathbf{x}}(k+1)-\mathbf{x}^{\ast}\right\|^{2}\right]\nonumber\\
	& \leq \underbrace{\exp\left(-\mu\sum_{l=0}^{k}\alpha_{l}\right)\mathbb{E}\left[\left\|\overline{\mathbf{x}}(k)-\mathbf{x}^{\ast}\right\|^{2}\right]}_{\text{$t_6$}}\nonumber\\
	&+\underbrace{\exp\left(-\mu\sum_{m=\lfloor\frac{k-1}{2}\rfloor}^{k}\alpha_{m}\right)\sum_{l=0}^{\lfloor\frac{k-1}{2}\rfloor-1}\frac{16NL^{2}\Delta_{1,\infty}\alpha_{0}^{3}}{\mu\lambda_{2}^{2}\left(\overline{\mathbf{R}}\right)c_0^{2}\beta_0^{2}(k+1)^{3-2\tau-2\delta}}}_{\text{$t_7$}}\nonumber\\&+\underbrace{\exp\left(-\frac{\mu}{N}\sum_{m=\lfloor\frac{k-1}{2}\rfloor}^{k}\alpha_{m}\right)\sum_{l=k_5}^{\lfloor\frac{k-1}{2}\rfloor-1}\frac{4M^{2}Nd^2(P)c_{0}^{4}\alpha_0}{\mu(k+1)^{1+4\delta}}}_{\text{$t_8$}}\nonumber\\
	&+\underbrace{\exp\left(-\mu\sum_{m=\lfloor\frac{k-1}{2}\rfloor}^{k}\alpha_{m}\right)\sum_{l=0}^{\lfloor\frac{k-1}{2}-1}P_{l}+\frac{2\alpha_{0}^{2}dc_{v}q_{\infty}(N,d,\alpha_0,c_0)}{c_0^2(l+1)^{2-2\delta}}}_{\text{$t_{10}$}}\nonumber\\
	&+\underbrace{\exp\left(-\mu\sum_{m=\lfloor\frac{k-1}{2}\rfloor}^{k}\alpha_{m}\right)\sum_{l=0}^{\lfloor\frac{k-1}{2}-1}\frac{2\alpha_{0}^{2}dN\sigma_1^{2}}{c_0^2(l+1)^{2-2\delta}}}_{\text{$t_{11}$}}\nonumber\\
	&+\underbrace{\frac{32NL^{2}\Delta_{1,\infty}\alpha_{0}^{2}}{\mu^2\lambda_{2}^{2}\left(\overline{\mathbf{R}}\right)c_0^{2}\beta_0^{2}(k+1)^{2-2\tau-2\delta}}}_{\text{$t_{12}$}}\nonumber\\
	&+\underbrace{\frac{8NM^{2}d^2(P)c_{0}^{4}}{\mu^{2}(k+1)^{4\delta}}}_{\text{$t_{13}$}}+\underbrace{\frac{N(k+1)P_{k}}{\mu\alpha_0}}_{\text{$t_{14}$}}\nonumber\\
	&+\underbrace{\frac{4N\alpha_{0}\left(dc_{v}q_{\infty}(N,d,\alpha_0,c_0)+dN\sigma_1^{2}\right)}{\mu c_0^2(k+1)^{1-2\delta}}}_{\text{$t_{15}$}}.
	\end{align}}
It is to be noted that the term $t_6$ decays as $1/k$. The terms $t_7$, $t_8$, $t_{10}$, $t_{11}$ and $t_{14}$ decay faster than its counterparts in the terms $t_{12}$, $t_{13}$ and $t_{15}$ respectively. We note that $Q_{l}$ also decays faster.
Hence, the rate of decay of $\mathbb{E}\left[\left\|\overline{\mathbf{x}}(k+1)-\mathbf{x}^{\ast}\right\|^{2}\right]$ is determined by the terms $t_{12}$, $t_{13}$ and $t_{15}$. Thus, we have that, $\mathbb{E}\left[\left\|\overline{\mathbf{x}}(k+1)-\mathbf{x}^{\ast}\right\|^{2}\right]=O\left(k^{-\delta_{1}}\right)$,
where $\delta_{1}=\min\left\{1-2\delta,2-2\tau-2\delta,4\delta\right\}$. For notational ease, we refer to $t_6+t_7+t_8+t_{10}+t_{11}+t_{14}= M_{k}$ from now on.
Finally, we note that,
{\small\begin{align}
\label{eq:opt12}
&\left\|\mathbf{x}_{i}(k)-\mathbf{x}^{\ast}\right\| \le \left\|\overline{\mathbf{x}}(k)-\mathbf{x}^{\ast}\right\|+\left\|\underbrace{\mathbf{x}_{i}(k)-\overline{\mathbf{x}}(k)}_{\text{$\widetilde{\mathbf{x}}_{i}(k)$}}\right\|\nonumber\\
&\Rightarrow\left\|\mathbf{x}_{i}(k)-\mathbf{x}^{\ast}\right\|^{2} \leq 2\left\|\widetilde{\mathbf{x}}_{i}(k)\right\|^{2}+2\left\|\overline{\mathbf{x}}(k)-\mathbf{x}^{\ast}\right\|^{2}\nonumber\\
&\Rightarrow\mathbb{E}\left[\left\|\mathbf{x}_{i}(k)-\mathbf{x}^{\ast}\right\|^{2}\right] \le 2M_{k}+\frac{64NL^{2}\Delta_{1,\infty}\alpha_{0}^{2}}{\mu^2\lambda_{2}^{2}\left(\overline{\mathbf{R}}\right)c_0^{2}\beta_0^{2}(k+1)^{2-2\tau-2\delta}}\nonumber\\&+\frac{16NM^{2}d^2(P)c_{0}^{4}}{\mu^{2}(k+1)^{4\delta}}+2Q_{k}+
\frac{8\Delta_{1,\infty}\alpha_{0}^{2}}{\lambda_{2}^{2}\left(\overline{\mathbf{R}}\right)\beta_0^{2}c_0^2(k+1)^{2-2\tau-2\delta}}\nonumber\\&+
\frac{4N\alpha_{0}\left(dc_{v}q_{\infty}(N,d,\alpha_0,c_0)+dN\sigma_1^{2}\right)}{\mu c_0^2(k+1)^{1-2\delta}}\nonumber\\
&\Rightarrow\mathbb{E}\left[\left\|\mathbf{x}_{i}(k)-\mathbf{x}^{\ast}\right\|^{2}\right] = O\left(\frac{1}{k^{\delta_{1}}}\right),~~\forall i,
\end{align}}
where $\delta_{1}=\min\left\{1-2\delta,2-2\tau-2\delta,4\delta\right\}$. By, optimizing over $\tau$ and $\delta$, we obtain that for $\tau=1/2$ and $\delta=1/6$,
{\small\begin{align*}
\mathbb{E}\left[\left\|\mathbf{x}_{i}(k)-\mathbf{x}^{\ast}\right\|^{2}\right] = O\left(\frac{1}{k^{\frac{2}{3}}}\right),~~\forall i.
\end{align*}}
The communication cost is given by,
{\small\begin{align*}
%\label{eq:comm_cost1}
\mathbb{E}\left[\sum_{t=1}^{k}\zeta_{t}\right]=O\left(k^{\frac{3}{4}+\frac{\epsilon}{2}}\right).
\end{align*}}
Thus, we achieve the communication rate to be,
{\small\begin{align}
\label{eq:0comm_rate}
\mathbb{E}\left[\left\|\mathbf{x}_{i}(k)-{\mathbf{x}^{\star}}\right\|^{2}\right]=O\left(\frac{1}{\mathcal{C}_{k}^{8/9-\zeta}}\right),
\end{align}}
where $\zeta$ can be arbitrarily small.

\section{Proof of the main result: First order optimization}
\label{sec:proof_main_res_1}
\begin{Lemma}
	\label{Lemma-MSS-BDD}
	Consider algorithm \eqref{eqn-alg-node-i}, and
	let the hypotheses of Theorem \ref{theorem-2} hold.
	Then, we have that for all $k=0,1,...,$ there holds:
	{\small\begin{align*}
	&\mathbb{E}\left[\left\|\mathbf{x}(k)-\mathbf{x}^{o}\right\|^2\right]\le q_{k_0}(N,\alpha_0)\nonumber\\&+\frac{\pi^{2}}{6}\alpha_0^2\left(2c_{u}N\left\|\mathbf{x}^{o}\right\|^{2}+N\sigma_u^{2}\right)+4\frac{\|\nabla F(\mathbf{x}^{o})\|^2}{\mu^2}\doteq q_{\infty}(N,\alpha_0),
	\end{align*}}
	where $\mathbb{E}\left[\left\|\mathbf{x}(k_2)-\mathbf{x}^{o}\right\|^{2}\right] \le q_{k_2}(N,\alpha_0)$, $k_2=\max\{k_0,k_1\}$, $k_0 = \inf\{k|\mu^{2}\alpha_{k}^2 < 1\}$ and $k_{1}=\inf\left\{k | \frac{\mu}{2} > 2c_{u}\alpha_{k}\right\}$.
\end{Lemma}
\noindent\begin{IEEEproof}
\noindent Proceeding as in the proof of Lemma \ref{Lemma0-MSS-BDD}, with $c_{k}=1$ and $\mathbf{b}(\mathbf{x}(k))=0$, we have that,
	$\forall k \geq \max\left\{k_0,k_1\right\}$,
	{\small\begin{align}
		\label{eq:1combine-4}
		&\mathbb{E}\left[\|\boldsymbol{\zeta}(k+1)\|^2 \right] \le \prod_{l=k_0}^{k}\left(1-\frac{\mu\alpha_{l}}{2}\right)\mathbb{E}\left[\|\boldsymbol{\zeta}(k_0)\|^2\right]\nonumber\\
		&+\frac{\pi^{2}}{6}\alpha_0^2\left(2c_{u}N\left\|\mathbf{x}^{o}\right\|^{2}+N\sigma_{u}^{2}\right)\nonumber\\&+4\frac{\|\nabla F(\mathbf{x}^{o})\|^2}{\mu^2}\nonumber\\
		&\mathbb{E}\left[\|\boldsymbol{\zeta}(k+1)\|^2 \right] \le q_{k_2}(N,\alpha_0)+\frac{\pi^{2}}{6}\alpha_0^2\left(2c_{u}N\left\|\mathbf{x}^{o}\right\|^{2}+N\sigma_{u}^{2}\right)\nonumber\\&+4\frac{\|\nabla F(\mathbf{x}^{o})\|^2}{\mu^2}\nonumber\\&\doteq q_{\infty}(N,\alpha_0),
		\end{align}}
	where $k_0 = \inf\{k|\mu^{2}\alpha_{k}^2 < 1\}$ and
	{\small\begin{align*}
	k_{1}=\inf\left\{k | \frac{\mu}{2} > 2c_{u}\alpha_{k}\right\}.
	\end{align*}}
	and $k_2=\max\{k_0,k_1\}$. It is to be noted that $k_1$ is necessarily finite as $\alpha_{k}\to 0$  as $k\to\infty$. 
	\noindent Hence, we have that $\mathbb{E}\left[\left\|\mathbf{x}(k+1)-\mathbf{x}^{o}\right\|^{2}\right]$ is finite and bounded from above, where $\mathbb{E}\left[\left\|\mathbf{x}(k_2)-\mathbf{x}^{o}\right\|^{2}\right] \le q_{k_2}(N,\alpha_0)$. From the boundedness of $\mathbb{E}\left[\left\|\mathbf{x}(k)-\mathbf{x}^{o}\right\|^2\right]$, we have also established the boundedness of $\mathbb{E}\left[\left\|\nabla F(\mathbf{x}(k))\right\|^2\right]$ and $\mathbb{E}\left[\left\|\mathbf{x}(k)\right\|^2\right]$. 
	
	\noindent With the above development in place, we can bound the variance of the noise process $\{\mathbf{v}(k)\}$ as follows:
	{\small\begin{align}
	\label{eq:1noise_variance_condition}
	&\mathbb{E}\left[\left\|\mathbf{u}(k)\right\|^{2}|\mathcal{S}_{k}\right]\le 2c_{u}q_{\infty}(N,\alpha_0)\nonumber\\&+2N\underbrace{\left(\sigma_u^{2}+\left\|\mathbf{x}^{\ast}\right\|^{2}\right)}_{\text{$\sigma_{1}^{2}$}}.
	\end{align}}
\noindent The proof of Lemma \ref{Lemma-MSS-BDD} is now complete.
\end{IEEEproof}
\noindent Recall the (hypothetically available) global average of nodes' estimates
$\overline{\mathbf{x}}(k)=\frac{1}{N}\sum_{i=1}^N \mathbf{x}_i(k)$,
and denote by
$\widetilde{\mathbf{x}}_i(k)
=\mathbf{x}_i(k) - \overline{\mathbf{x}}(k)$ the quantity
that measures how far apart is node $i$'s
solution estimate from the global average.
Introduce also vector
$\widetilde{\mathbf{x}}(k)=(\,\widetilde{\mathbf{x}}_1(k),...,\widetilde{\mathbf{x}}_N(k)\,)^\top$,
and note that it can be represented as  $\widetilde{\mathbf{x}}(k)=\left(\mathbf{I}-\mathbf{J}\right)\mathbf{x}(k)$, where we recall $\mathbf{J}=\frac{1}{N}\mathbf{1}\mathbf{1}^{\top}$.
We have the following Lemma.
\begin{Lemma}
	\label{lemma-disag-bound}
	Let the hypotheses of Theorem \ref{theorem-2} hold. Then, we have
	{\small\begin{align*}
	&\mathbb{E}\left[\left\|\widetilde{\mathbf{x}}(k+1)\right\|^{2}\right] \le Q_{k}+ \frac{2\Delta_{1,\infty}\alpha_{0}^{2}}{\lambda_{2}^{2}\left(\overline{\mathbf{R}}\right)\beta_0^{2}(k+1)}\nonumber\\&=O\left(\frac{1}{k}\right),
	\end{align*}}
	where $Q_k$ is a term which decays faster than $(k+1)^{-1}$.
\end{Lemma}

Lemma~\ref{lemma-disag-bound}
is important as it allows
to sufficiently tightly
bound the bias in the
gradient estimates
according to which the
global average
$\overline{\mathbf{x}}(k)$ evolves.

\noindent\begin{IEEEproof}
Proceeding as in the proof of Lemma \ref{lemma-0disag-bound} in \eqref{eq:0dis1}-\eqref{eq:0dis5_bound}, we have,
	{\small\begin{align}
	\label{eq:1dis6}
	&\mathbb{E}\left[\left\|\widetilde{\mathbf{x}}(k+1)\right\|^{2}|\mathcal{S}_{k}\right]\leq \left(1+\theta_k\right)\left(1-\beta_{k}\lambda_{2}\left(\overline{\mathbf{R}}\right)\right)\left\|\widetilde{\mathbf{x}}(k)\right\|^{2}\nonumber\\&+\left(1+\frac{1}{\theta_k}\right)\alpha_k^2\mathbb{E}\left[\left\|\mathbf{w}(k)\right\|^{2}|\mathcal{F}_{k}\right],
	\end{align}}
	where 
	{\small\begin{align}
	\label{eq:1wk_bound}
	&\mathbb{E}\left[\left\|\mathbf{w}(k)\right\|^{2}|\mathcal{S}_{k}\right] \le 2\left\|\nabla F(\mathbf{x}(k))\right\|^{2}+2\mathbb{E}\left[\left\|\mathbf{v}(k)\right\|^{2}|\mathcal{F}_{k}\right]\nonumber\\
	&\le \underbrace{2\left\|\nabla F(\mathbf{x}(k))\right\|^{2}+4c_{u}q_{\infty}(N,\alpha_0)+4N\sigma_{1}^{2}}_{\text{$\Delta_{1,\infty}$}}\nonumber\\&\Rightarrow \mathbb{E}\left[\left\|\mathbf{w}(k)\right\|^{2}\right] < \infty.
	\end{align}}
	With the above development in place, we then have,
	{\small\begin{align}
	\label{eq:1dis6.5}
	&\mathbb{E}\left[\left\|\widetilde{\mathbf{x}}(k+1)\right\|^{2}\right]\leq \left(1+\theta_k\right)\left(1-\beta_{k}\lambda_{2}\left(\overline{\mathbf{R}}\right)\right)\left\|\widetilde{\mathbf{x}}(k)\right\|^{2}\nonumber\\&+\left(1+\frac{1}{\theta_k}\right)\alpha_k^2\Delta_{1,\infty}.
	\end{align}}
	In particular, we choose $\theta(k) = \frac{\beta_{k}}{2}\lambda_{2}\left(\overline{\mathbf{R}}\right)$. From \eqref{eq:dis6.5}, we have,
	{\small\begin{align}
	\label{eq:1dis6.6}
	&\mathbb{E}\left[\left\|\widetilde{\mathbf{x}}(k+1)\right\|^{2}\right]\leq \left(1-\frac{\beta_{k}}{2}\lambda_{2}\left(\overline{\mathbf{R}}\right)\right)\mathbb{E}\left[\left\|\widetilde{\mathbf{x}}(k)\right\|^{2}\right]\nonumber\\&+\left(1+\frac{2}{\beta_{k}\lambda_{2}\left(\overline{\mathbf{R}}\right)}\right)\alpha_k^2\Delta_{1,\infty}\nonumber\\
	&= \left(1-\frac{\beta_{k}}{2}\lambda_{2}\left(\overline{\mathbf{R}}\right)\right)\mathbb{E}\left[\left\|\widetilde{\mathbf{x}}(k)\right\|^{2}\right]\nonumber\\&+\frac{2\alpha_k^2}{\lambda_{2}\left(\overline{\mathbf{R}}\right)\beta_{k}}\Delta_{1,\infty}+\alpha_k^2\Delta_{1,\infty}.
	\end{align}}
	\noindent Applying lemma \ref{lemma:technical1} to $q_k=\mathbb{E}\left[\left\|\widetilde{\mathbf{x}}(k)\right\|^{2}\right]$, $d_k=\Delta_k$, $b_k=\alpha_k^2$, and $s_k=\frac{\beta_{k}}{2}\lambda_{2}\left(\overline{\mathbf{R}}\right)$, we obtain for $m(k)= \lfloor\frac{k-1}{2}\rfloor$:
		{\small\begin{align}
		\label{eq:1dis6.7}
		&\mathbb{E}\left[\left\|\widetilde{\mathbf{x}}(k+1)\right\|^{2}\right]\leq 	 \underbrace{\exp\left(-\sum_{l=0}^{k}s(l)\right)\mathbb{E}\left[\left\|\widetilde{\mathbf{x}}(0)\right\|^{2}\right]}_{\text{$t_1$}}\nonumber\\
		&+\underbrace{\Delta_{1,\infty}\exp\left(-\sum_{m=\lfloor\frac{k-1}{2}\rfloor}^{k}s(m)\right)\sum_{l=0}^{\lfloor\frac{k-1}{2}\rfloor-1}\left(\frac{2\alpha_l^2}{\lambda_{2}\left(\overline{\mathbf{R}}\right)\beta_{l}}+\alpha_l^2\right)}_{\text{$t_2$}}\nonumber\\
		&+\underbrace{\frac{2\Delta_{1,\infty}\alpha_{0}^{2}}{\lambda_{2}^{2}\left(\overline{\mathbf{R}}\right)\beta_0^{2}(k+1)}}_{\text{$t_3$}}+\underbrace{\frac{4\Delta_{1,\infty}\alpha_{0}^{2}}{\lambda_{2}\left(\overline{\mathbf{R}}\right)\beta_0(k+1)^{3/2}}}_{\text{$t_4$}}.
		\end{align}}
		In the proof of Lemma~\ref{lemma:technical1}, the splitting in the interval $[0,k]$ was done at $\lfloor\frac{k-1}{2}\rfloor$ for ease of book keeping. 
	The division can be done at an arbitrary point.
	It is to be noted that the sequence $\{s(k)\}$ is not summable and hence terms $t_1$ and $t_2$ decay faster than $(k+1)$. Also, note that term $t_4$ decays faster than $t_3$. 
	For notational ease, henceforth we refer to $t_1+t_2+t_4=Q_{k}$, while keeping in mind that $Q_{k}$ decays faster than $(k+1)$.
	%\end{proof}
	Hence, we have the disagreement given by,
	{\small\begin{align*}
	\mathbb{E}\left[\left\|\widetilde{\mathbf{x}}(k+1)\right\|^{2}\right] =O\left(\frac{1}{k}\right).
	\end{align*}}
\end{IEEEproof}
\begin{Lemma}
	\label{lemma-comm-rate}
	Consider algorithm \eqref{eqn-alg-node-i}
	and let the hypotheses of Theorem \ref{theorem-2} hold. Then, there holds:
	\[
	\mathbb{E}[\,\|\mathbf{x}_{i}(k)-\mathbf{x}^{\star}\|^2\,] = O(1/k).
	\] and
	\[
	\mathbb{E}[\,\|\mathbf{x}_{i}(k)-\mathbf{x}^{\star}\|^2\,] = O\left(\frac{1}{\mathcal{C}_{k}^{4/3-\zeta}}\right),
	\]
	where $\zeta > 0$ can be arbitrarily small, for all $i=1,\cdots,N$.
\end{Lemma}
\begin{IEEEproof}
	Denote $\overline{\mathbf{x}}(k)=\frac{1}{N}\sum_{n=1}\mathbf{x}_{i}(k)$.
	Then, we have,\\
	{\small\begin{align}
		\label{eq:1opt1}
		\overline{\mathbf{x}}(k+1) = \overline{\mathbf{x}}(k)-\alpha_{k}\left[\frac{1}{N}\sum_{i=1}^{N}\nabla f_{i}\left(\mathbf{x}_{i}(k)\right)+\underbrace{\frac{1}{N}\sum_{i=1}^{N}
			\mathbf{u}_{i}(k)}_{\text{$\overline{\mathbf{u}}(k)$}}\right]%,\nonumber
		\end{align}}
	which implies:
	{\small\begin{align*}
	& \overline{\mathbf{x}}(k+1) = \overline{\mathbf{x}}(k)-\frac{\alpha_k}{N}\left[\sum_{i=1}^{N}\nabla f_{i}\left(\mathbf{x}_{i}(k)\right) \right.\\
	&\left. -\nabla f_{i}\left(\overline{\mathbf{x}}(k)\right)+\nabla f_{i}\left(\overline{\mathbf{x}}(k)\right)\right]-\alpha_{k}\overline{\mathbf{u}}(k).
	\end{align*}}
	where
	{\small\begin{align}
	\label{eq:1opt3}
	&\mathbf{e}(k) = N\overline{\mathbf{u}}(k)\nonumber\\&+\underbrace{\sum_{i=1}^{N}\left(\nabla f_{i}\left(\mathbf{x}_{i}(k)\right)-\nabla f_{i}\left(\overline{\mathbf{x}}(k)\right)\right)}_{\text{$\boldsymbol{\epsilon}(k)$}}.
	\end{align}}
	Proceeding as in \eqref{eq:opt4}-\eqref{eq:opt10}, with $c_{k}=1$ and $\mathbf{b}(\mathbf{x}_{i}(k))=0$, $\forall i=1,\cdots,N$, we have
	on choosing $\theta_{k}=\frac{\mu\alpha_0}{k+1}$, where $\alpha_0 > \frac{1}{\mu}$, 
	{\small\begin{align}
		\label{eq:1opt11}
		&\mathbb{E}\left[\left\|\mathbf{m}(k)\right\|^{2}\right] \leq \left(1-\frac{\mu\alpha_0}{k+1}\right)\mathbb{E}\left[\left\|\overline{\mathbf{x}}(k)-\mathbf{x}^{\ast}\right\|^{2}\right]\nonumber\\&+
		\frac{8NL^{2}\Delta_{1,\infty}\alpha_{0}^{3}}{\mu\lambda_{2}^{2}\left(\overline{\mathbf{R}}\right)\beta_0^{2}(k+1)^{2}}+\frac{2NL^{2}Q_{k}}{\mu(k+1)}\nonumber\\
		&\Rightarrow \mathbb{E}\left[\left\|\overline{\mathbf{x}}(k+1)-\mathbf{x}^{\ast}\right\|^{2}\right] \leq \left(1-\frac{\mu\alpha_0}{k+1}\right)\mathbb{E}\left[\left\|\overline{\mathbf{x}}(k)-\mathbf{x}^{\ast}\right\|^{2}\right]\nonumber\\&+\frac{8NL^{2}\Delta_{1,\infty}\alpha_{0}^{3}}{\mu\lambda_{2}^{2}\left(\overline{\mathbf{R}}\right)\beta_0^{2}(k+1)^{2}}+\frac{2NL^{2}Q_{k}}{\mu(k+1)}+2\alpha_{k}^{2}\left(c_{u}q_{\infty}(N,\alpha_0)+N\sigma_1^{2}\right)\nonumber\\
		&\Rightarrow\mathbb{E}\left[\left\|\overline{\mathbf{x}}(k+1)-\mathbf{x}^{\ast}\right\|^{2}\right] \leq \left(1-\frac{\mu\alpha_0}{k+1}\right)\mathbb{E}\left[\left\|\overline{\mathbf{x}}(k)-\mathbf{x}^{\ast}\right\|^{2}\right]\nonumber\\&+\frac{8NL^{2}\Delta_{1,\infty}\alpha_{0}^{3}}{\mu\lambda_{2}^{2}\left(\overline{\mathbf{R}}\right)\beta_0^{2}(k+1)^{2}}+2\alpha_{k}^{2}\left(c_{u}q_{\infty}(N,\alpha_0)+N\sigma_1^{2}\right)+P_{k},
		\end{align}}
	where $P_{k}$ decays faster as compared to the other terms.
	Proceeding as in \eqref{eq:dis6.7}, we have
	{\small\begin{align}
		\label{eq:1opt11.5}
		&\mathbb{E}\left[\left\|\overline{\mathbf{x}}(k+1)-\mathbf{x}^{\ast}\right\|^{2}\right]\nonumber\\
		& \leq \underbrace{\exp\left(-\mu\sum_{l=0}^{k}\alpha_{l}\right)\mathbb{E}\left[\left\|\overline{\mathbf{x}}(k)-\mathbf{x}^{\ast}\right\|^{2}\right]}_{\text{$t_6$}}\nonumber\\
		&+\underbrace{\exp\left(-\mu\sum_{m=\lfloor\frac{k-1}{2}\rfloor}^{k}\alpha_{m}\right)\sum_{l=0}^{\lfloor\frac{k-1}{2}\rfloor-1}\frac{8L^{2}\Delta_{1,\infty}\alpha_{0}^{3}}{\mu\lambda_{2}^{2}\left(\overline{\mathbf{R}}\right)\beta_0^{2}(l+1)^{2}}}_{\text{$t_7$}}\nonumber\\
		&+\underbrace{\exp\left(-\mu\sum_{m=\lfloor\frac{k-1}{2}\rfloor}^{k}\alpha_{m}\right)\sum_{l=0}^{\lfloor\frac{k-1}{2}-1}P_{l}}_{\text{$t_{10}$}}\nonumber\\
		&+\underbrace{\exp\left(-\mu\sum_{m=\lfloor\frac{k-1}{2}\rfloor}^{k}\alpha_{m}\right)\sum_{l=0}^{\lfloor\frac{k-1}{2}-1}\frac{2\alpha_{0}^{2}\left(c_{u}q_{\infty}(N,\alpha_0)+N\sigma_1^{2}\right)}{(l+1)^{2}}}_{\text{$t_{11}$}}\nonumber\\
		&+\underbrace{\frac{16NL^{2}\Delta_{1,\infty}\alpha_{0}^{2}}{\mu^2\lambda_{2}^{2}\left(\overline{\mathbf{R}}\right)\beta_0^{2}(k+1)}}_{\text{$t_{12}$}}\nonumber\\
		&+\underbrace{\frac{N(k+1)P_{k}}{\mu\alpha_0}}_{\text{$t_{14}$}}+\underbrace{\frac{4N\alpha_{0}\left(c_{u}q_{\infty}(N,\alpha_0)+N\sigma_1^{2}\right)}{\mu(k+1)}}_{\text{$t_{15}$}}.
		\end{align}}
	It is to be noted that the term $t_6$ decays as $1/k$. The terms $t_7$, $t_{10}$ and $t_{11}$ decay faster than its counterparts in the terms $t_{12}$ and $t_{15}$ respectively. We note that $Q_{l}$ also decays faster.
	Hence, the rate of decay of $\mathbb{E}\left[\left\|\overline{\mathbf{x}}(k+1)-\mathbf{x}^{\ast}\right\|^{2}\right]$ is determined by the terms $t_{12}$ and $t_{15}$. Thus, we have that, $\mathbb{E}\left[\left\|\overline{\mathbf{x}}(k+1)-\mathbf{x}^{\ast}\right\|^{2}\right]=O\left(\frac{1}{k}\right)$. For notational ease, we refer to $t_6+t_7+t_{10}+t_{11}+t_{14}= M_{k}$ from now on.
	Finally, we note that,
	{\small\begin{align}
	\label{eq:1opt12}
	&\left\|\mathbf{x}_{i}(k)-\mathbf{x}^{\ast}\right\| \le \left\|\overline{\mathbf{x}}(k)-\mathbf{x}^{\ast}\right\|+\left\|\underbrace{\mathbf{x}_{i}(k)-\overline{\mathbf{x}}(k)}_{\text{$\widetilde{\mathbf{x}}_{i}(k)$}}\right\|\nonumber\\
	&\Rightarrow\left\|\mathbf{x}_{i}(k)-\mathbf{x}^{\ast}\right\|^{2} \leq 2\left\|\widetilde{\mathbf{x}}_{i}(k)\right\|^{2}+2\left\|\overline{\mathbf{x}}(k)-\mathbf{x}^{\ast}\right\|^{2}\nonumber\\
	&\Rightarrow\mathbb{E}\left[\left\|\mathbf{x}_{i}(k)-\mathbf{x}^{\ast}\right\|^{2}\right] \le 2M_{k}+\frac{32NL^{2}\Delta_{1,\infty}\alpha_{0}^{2}}{\mu^2\lambda_{2}^{2}\left(\overline{\mathbf{R}}\right)\beta_0^{2}(k+1)}\nonumber\\&+2Q_{k}+ \frac{4\Delta_{1,\infty}\alpha_{0}^{2}}{\lambda_{2}^{2}\left(\overline{\mathbf{R}}\right)\beta_0^{2}(k+1)}\nonumber\\
	&\Rightarrow\mathbb{E}\left[\left\|\mathbf{x}_{i}(k)-\mathbf{x}^{\ast}\right\|^{2}\right] = O\left(\frac{1}{k}\right),~~\forall i.
	\end{align}}
	The communication cost is given by,
	{\small\begin{align*}
	%\label{eq:comm_cost1}
	\mathbb{E}\left[\sum_{t=1}^{k}\zeta_{t}\right]=O\left(k^{\frac{3}{4}+\frac{\epsilon}{2}}\right).
	\end{align*}}
	Thus, we achieve the communication rate to be,
	{\small\begin{align}
	\label{eq:12comm_rate}
	\mathbb{E}\left[\left\|\mathbf{x}_{i}(k)-{\mathbf{x}^{\star}}\right\|^{2}\right]=O\left(\frac{1}{\mathcal{C}_{k}^{\frac{4}{3}-\zeta}}\right).
	\end{align}}

\end{IEEEproof}
\section{Conclusion}
\label{sec:conc}
\noindent In this paper, we have developed and analyzed a novel class of methods for distributed stochastic optimization of
the zeroth and first order that are based on increasingly sparse randomized communication protocols.
We have established for both the proposed zeroth and first order methods explicit mean square error (MSE) convergence rates with respect to (appropriately defined) computational cost $C_{\mathrm{comp}}$ and communication cost $C_{\mathrm{comm}}$. Specifically,
the proposed zeroth order method achieves the $O(1/(C_{\mathrm{comm}})^{8/9-\zeta})$ MSE communication rate, which
significantly improves over the rates of existing methods, while maintaining the order-optimal
$O(1/(C_{\mathrm{comp}})^{2/3})$ MSE computational rate. Similarly, the proposed first order method achieves
the $O(1/(C_{\mathrm{comm}})^{4/3-\zeta})$ MSE communication rate, while maintaining the order-optimal
$O(1/(C_{\mathrm{comp}}))$ MSE computational rate.  Numerical examples on real data demonstrate the communication efficiency of the proposed methods.
\bibliographystyle{IEEEtran}
%%\bibliography{IEEEabrv,CentralBib}
\bibliography{IEEEabrv,bibliographyJournal,dsprt,glrt,CentralBib}

\end{document}